\documentclass[leqno]{amsart}
\usepackage{amsmath}        
\usepackage{amsfonts}
\usepackage{enumerate}
\usepackage{amsthm}
\usepackage{amssymb}
\usepackage{amscd}
\usepackage[all]{xy}
\usepackage[titletoc,toc,title]{appendix}

\newtheorem{thm}{Theorem}[section] 

\newtheorem{prop}[thm]{Proposition}
\newtheorem{cor}[thm]{Corollary}
\newtheorem{lemma}[thm]{Lemma}

\theoremstyle{definition}

\theoremstyle{definition}
\newtheorem*{Def}{Definition}
\newtheorem*{rem}{Remark}
\newtheorem*{rems}{Remarks}
\newtheorem{Rem}[thm]{Remark}

\newtheorem{example}[thm]{Example}
\newtheorem{examples}[thm]{Examples}
\newtheorem{conj}[thm]{Conjecture}

\def\sk1{\vskip 10pt}
\def\newline{\hfil\break}

\def\ds{\displaystyle}
\def\eset{\emptyset}

\def\subeq{\subseteq}
\def\supeq{\supseteq}
\def\Sigmacirc {{^{\circ} \Sigma}}
\def\Sigmadoublecirc {{^{\circ \circ} \Sigma}}
\def\thra{\twoheadrightarrow}

\def\Z{\mathbb Z}
\def\R{\mathbb R}
\def\M{\mathbb M}
\def\C{\mathbb C}
\def\H{\mathbb H}
\def\Q{\mathbb Q}
\def\E{\mathbb E}
\def\N{\mathbb N}

\def\del{\partial}

\def\sk{\vskip}
\newcommand{\Ge}{Ge}
\newcommand{\Gr}{Gr}

\numberwithin{equation}{section}

\title{Limit sets for modules over groups on $CAT(0)$ spaces -- from the Euclidean to the hyperbolic}

\author{Robert Bieri and Ross Geoghegan}

\address{\noindent Robert Bieri, Fachbereich Mathematik,
Johann Wolfgang Goethe-Universit\"at Frankfurt,
D-60054 Frankfurt am Main,
Germany \newline
\vskip 1pt
and
\newline
\vskip 1pt
Department of Mathematical Sciences,
Binghamton University (SUNY),
Binghamton, NY 13902-6000, USA
\newline
\vskip 3pt
Ross Geoghegan, Department of Mathematical Sciences,
Binghamton University (SUNY),
Binghamton, NY 13902-6000, USA}
\email{bieri@math.uni-frankfurt.de,
ross@math.binghamton.edu}


\date{October 31, 2016}


\begin{document}
\fontsize{12}{13pt} \selectfont  

\begin{abstract}
The observation that the $0$-dimensional Geometric Invariant $\Sigma
^{0}(G;A)$ of Bieri-Neumann-Strebel-Renz can be interpreted as a
horospherical limit set opens a direct trail from Poincar\'{e}'s
limit set $\Lambda (\Gamma )$ of a discrete group $\Gamma $ of
M\"{o}bius transformations (which contains the horospherical limit
set of $\Gamma $) to the roots of tropical geometry (closely related
to $\Sigma ^{0}(G;A)$ when $G$ is abelian).  We explore this trail by
introducing the horospherical limit set, $\Sigma (M;A)$, of a $G$-module
$A$ when $G$ acts by isometries on a proper $CAT(0)$ metric space $M$.
This is a subset of the boundary at infinity $\del M$. On the way we meet
instances where $\Sigma (M;A)$ is the set of all conical limit points ($G$ geometrically finite and $M$ hyperbolic), the complement of the radial projection of a tropical variety ($G$ abelian and $M$ Euclidean) or the complement of a spherical building ($G$ arithmetic and $M$ symmetric).
\end{abstract}

\maketitle

\tableofcontents

\section{Introduction}\label{intro}

\subsection{Horospherical limit sets of $G$-modules.}\label{Section1}

Consider a discrete group $G$ and two $G$-objects: a finitely generated
$G$-module $A$ and a proper $G$-$CAT(0)$ space $M$ (i.e. $M$
comes equipped with a $G$ action by isometries). While {\it a priori}
there is no connection between $M$ and $A$, there is much to
be learned by ``controlling" $A$ over $M$. This control is
achieved by means of a $G$-map

$$L:A\to \{\text{subsets of }\del M\}$$ 

\noindent where $\del M$ is the boundary of $M$ at infinity; $\del M$
carries the topological $G$-action induced by the action on $M$. To define
$L$ we choose a free presentation $\epsilon:F\twoheadrightarrow A$, where
$F$ is the free $G$-module over the finite basis $X$, and a base point
$b\in M$.  The {\it support} of $c\in F$, $\text{supp}(c)\subseteq G$,
is the (finite) set of elements of $G$ occurring in the unique expansion
of $c$ over $GX$.  Putting $h(c):=\text{supp}(c)b\subseteq M$ defines a
$G$-map $h:F\to fM$, from $F$ into the $G$-set $fM$ of all finite subsets
of $M$. We call $h$ a {\it control map} and define the {\it horospherical
mit points} of $a\in A$ to be the points in $$L(a):=\{e\in \del M\mid
\text{every horoball at } e\text{ contains } h(c) \text { for some
}c\text{ such that }\epsilon (c)=a\}$$ \noindent We show that $L(a)$
is independent of the choice of $\epsilon :F\twoheadrightarrow A$ and of
the base point $b$. By the {\it horospherical limit set of} $A$ we mean
the set $\Sigma (M;A)$ of points $e\in \del M$ which are limit points
of every element of $A$: i.e. $\Sigma (M;A):=\bigcap _{a\in A}L(a)$.
It is this set, together with its close relatives to be defined below,
which is our subject, both in general and via specific instances.

\subsection{Two substantial general results on the $G$-pairs $(M,A)$}

Since the paper is concerned with the horospherical limit set $\Sigma
(M;A)\subseteq \del M$, a natural first question is: when does it happen
that $\Sigma (M;A)=\del M$? On this we have a definitive answer\footnote{Note that $G$ is not assumed to be finitely generated, and no
assumption of discreteness, cocompactness or faithfulness of its action
on $M$ is made.}:

\begin{thm}\label{full}Assume $A\neq 0$. Then $\Sigma (M;A)=\del M$ if and only if the $G$-action
on $M$ is cocompact and $A$ is supported over a bounded subset $B$ of $M$
(i.e. for each $a\in A$  there is some $c\in F$ with $\epsilon (c)=a$
and $h(c)\subseteq B$).  \end{thm}

This has an algebraic meaning in the discrete-orbit case:

\begin{cor} If the $G$-action on $M$ is cocompact with discrete orbits then
$\Sigma (M;A) =\del M$ if and only if $A$ is finitely generated as a
module over some (equivalently any) point stabilizer.  
\end{cor}

Another natural question is: what kind of openness properties does $\Sigma (M;A)$ have? One of our openness theorems deals with what happens when the $G$-action on $M$ is perturbed:

\begin{thm}\label{open}If $\Sigma (M;A)=\del M$ holds for a given action
$\rho:G\to \text{\rm Isom}(M)$ then it holds in a neighborhood of $\rho $
in the space $\text{\rm Hom}(G, \text{\rm Isom}(M))$ of isometric actions.
\end{thm}

Again, the conclusion becomes more recognizable in the discrete-orbit case:

\begin{cor}\label{discretecase} 
\text{\rm (a)}
The set of cocompact actions of $G$ on $M$ by isometries is open.
\text{\rm (b)}
Let ${\mathcal R}(G,M)\subseteq \text{\rm Hom}(G, \text{\rm Isom}(M))$ be the
subspace of all actions such that $\rho (G)$ is cocompact with discrete
orbits and finite stabilizers.  For every finitely generated $G$-module
$A$ the set 
$$\{\rho \mid A \text{ \rm is finitely generated over ker}(\rho )\}$$
is open in ${\mathcal R}(G,M)$.
\end{cor}

\subsection{A unifying concept.}

As intriguing and motivating as such general results may be, there is the fact that
the horospherical limit set shows up in a number of highly interesting mathematical contexts. Among these are:

\begin{enumerate}[(1)]

\item If $G$ acts with discrete orbits on $M$ and trivially on $A\neq 0$ then our limit set  $\Sigma (M;A)\subseteq \del M$ is the usual horospherical limit set of the orbits. A variant of $\Sigma (M;A)$ which we denote by $\Lambda (M;A)$ coincides with the classical limit set
of $G$ in that case.

\item If $\bf G$ is the Lie group $\text{SL}(n,{\R})$ acting on its
symmetric space $M=\text{SL}(n,{\R})\slash \text{SO}(n)$ then the
building at infinity $B({\bf G})$ comes with a natural surjection
$\pi:B({\bf G})\twoheadrightarrow \del M$, compatible with the Tits
metric, and in that case the horospherical limit set of $\bf G$ with
respect to the trivial $\bf G$-module $\Z$ is the whole of $\del
M$. When we restrict the action to the arithmetic subgroup $G =
\text{SL}(n,{\Z})$, the horospherical limit set is much smaller: $\Sigma
(M;{\Z})\subseteq  \del M$ is now the complement of the image $\pi
(B(G))\subseteq \del M$ of the rational building. This was a conjecture
of Hanno Rehn who studied the higher homotopical version of our limit
sets \cite{BGe03} in his thesis \cite{Reh09} in the case when $G =
\text{SL}(n,{\Z}[\frac{1}{m}])$. It has recently been
proved for these (and more general) arithmetic groups by Avramidi
and Witte-Morris \cite{AWM}.

\item In the case where $M$ is Euclidean and $G$ acts by a
discrete translation group, $\Sigma (M;A)$ is the $0$-dimensional part of the
Bieri-Neumann-Strebel-Renz invariant, a group theoretic tool for
questions related to homological finiteness properties of infinite
groups and their modules. (In fact, the present work grew out of
our aim to extend the leading ideas of the BNSR theory to the $CAT(0)$ case). 

\item The special case where $M$ is Euclidean and $G$ is abelian, is intimately
connected with tropical geometry. Here, the complement of our
horospherical limit set appears as the radial projection of the integral
tropical variety associated to the annihilator ideal of the $G$-module
$A$. This is discussed in the Appendix.
\end{enumerate}

\subsection {Dynamical limit points and finitary homomorphisms}

In order to prove Theorems \ref{full} and \ref{open} we had to
consider subsets of $\Sigma (M;A)$ with a dynamical flavor. This required
measuring the quality of the convergence of a sequence of finite subsets
of $M$ towards a boundary point $e\in \del M$ in terms of the Busemann
function\footnote{Our convention is that $\beta _{e}(b)=0$ for a fixed
base point $b\in M$, and $\beta _{e}(e)=+\infty$} $\beta _{e}:M\to {\R}$.

Let $\epsilon:F\twoheadrightarrow A$ be the free presentation of
Section \ref{Section1}. We say that $e\in \del M$ is an {\it equivariant-dynamical
limit point} of the pair $(M,A)$ if there is a $G$-endomorphism
$\varphi :F\to F$ which induces the identity of $A$ and has the property
that there is a number $\delta >0$ with 
$\text{min}\beta_{e}(h(\varphi (c)))\geq \text{min}\beta_{e}(h(c))+\delta $ for all $c\in F$. 
Note that the sets $h(\varphi ^{i}(c))$ exhibit $e$ as a horospherical limit point of $\epsilon (c)$,
hence the {\it equivariant-dynamical limit set}
$$\Sigmadoublecirc (M; A):=\{e\in \del M\mid e  \text{ is an equivariant-dynamical limit point} \text{ of }(M,A)\}$$

\noindent is a subset of $\Sigma (M;A)$. In fact, each $e\in \Sigma
(M;A)$ could be called an additive-dynamical limit point, for it is easy
to exhibit $e$ as the limit set of a sequence $h(\varphi ^{i}(c))$
for some additive endomorphism $\varphi :F\to F$ as above --- see
Section \ref{push}. This concept is not new: in the case of a Euclidean
discrete action it was a crucial lemma (\cite{BS80}, \cite{BNS87} and
\cite{BRe88}) that $\Sigmadoublecirc (M; A)=\Sigma (M; A)$. But in the
general $CAT(0)$ case $\Sigmadoublecirc (M; A)$ is often dramatically
smaller than $\Sigma (M; A)$ as we can see from:

\begin{thm}\label{double} If $e\in \Sigmadoublecirc (M;A)$ then the closure (in the cone topology) of the orbit $Ge$, $\text {\rm cl}_{\partial M}(Ge)$, lies in a Tits-metric ball with radius $r<\frac {\pi}{2}$, and the center of the unique minimal ball with this property is fixed by $G$.
\end{thm}

The usefulness of this fixed point theorem is somewhat reduced by
the fact that the most interesting group actions (for example, any non-elementary
Fuchsian group on the hyperbolic plane) do not have fixed points in
$\del M$. Thus $\Sigmadoublecirc $ will be empty in these cases and
hence cannot be a useful tool for actions on hyperbolic spaces. This
suggested that considering dynamical limit sets only when $\varphi :F\to
F$ is a $G$-endomorphism is too restrictive. Instead  we had to find
a class of additive endomorphisms $\varphi :F\to F$ more flexible than
$G$-endomorphisms but still sharing some of their coarse features.

\begin{Def}($G$-{\it finitary homomorphisms}). An additive homomorphism $\varphi :A\to B$ between $G$-modules
is $G$-{\it finitary} if there is a $G$-map $\Phi :A\to fB$ of the
$G$-set underlying $A$ into the $G$-set $fB$ of all finite subsets of $B$
with the property that $\varphi (a)\in \Phi(a)$ for every $a\in A$.
We say that $\varphi $ is a {\it selection} from the $G$-{\it volley} $\Phi$. 
\end{Def}

\noindent {\bf Definition.} We say that $e$ is a {\it finitary-dynamical limit point} of the pair $(M,A)$ if there is a $G$-finitary endomorphism $\varphi :F\to F$,
which induces the identity of $A$ and has the property that there is a number $\delta >0$ with
$\text{min}\beta_{e}(h(\varphi (c)))\geq \text{min}\beta_{e}(h(c))+\delta
$.  
\noindent for all $c\in F$. 
The {\it finitary-dynamical limit set}
$$\Sigmacirc (M;A):=\{e\in \del M\mid e  \text{ is a finitary-dynamical limit point} \text{ of }(M,A)\}$$
\noindent is the main technical tool of the paper. Clearly, equivariant $\implies $ finitary $\implies $ additive, so we have $\Sigmadoublecirc \subseteq \Sigmacirc \subseteq \Sigma $. 

The precise relationship of $\Sigmacirc (M;A)$ to $\Sigma (M;A)$ is

\begin{thm}\label{circle} $\Sigmacirc (M;A)$ is a $G$-invariant subset of $\Sigma
(M;A)$; it consists of all $e\in \Sigma (M,A)$ with the property that $\text{\rm cl}(Ge)$,
the cone-topology-closure of the $G$-orbit of $e$, is contained in $\Sigma
(M;A)$. In particular $\Sigmacirc (M;A)$ contains every closed
$G$-invariant subset of $\Sigma (M;A)$, and hence 
$\Sigmacirc (M;A)=\del M$ if and only if $\Sigma (M;A)=\del M$.  
\end{thm}

\subsection{$\Sigmacirc (M;A)$ as an object of interest in its own right.}

\begin{enumerate}[(1)]

\item $\Sigmacirc (M;A)$ combines the cone topology and the Tits
metric topology in an interesting way: on the one hand  if $e\in \Sigmacirc (M;A)$ then
$\Sigmacirc (M;A)$ contains not only the orbit $Ge$ but also
the cone-topology closure of that orbit (Theorem \ref{circle}); on the other hand we have:

\begin{thm} $\Sigmacirc (M;A)$  and $\Sigmadoublecirc (M;A)$ are open in the Tits metric topology\footnote{See Theorem \ref{open1}.} on $\del M$.
\end{thm}

\item In the $G$-finitary category of $G$-modules\footnote{The category whose objects are $G$-modules and whose morphisms are $G$-finitary additive maps} the Fundamental Theorem of Homological Algebra holds true: every $G$-finitary homomorphism between two modules $A$ and $B$ can be lifted to a $G$-finitary chain map between the projective resolutions of $A$ and $B$, and any two lifts are homotopic by a $G$-finitary homotopy. That is precisely what we need to extend the definition of   $\Sigmacirc (M;A)$ to higher-dimensional invariants ${^{\circ} \Sigma}^{n}(M;A)$ when $A$ admits a free resolution with finite $n$-skeleton,  and to prove our openness results for those. This will appear in a subsequent paper.

\item For each $e\in \del M$ we consider the set $\widehat {{\Z}G^{e}}$
of all formal sums $\Sigma _{g\in G}n_{g}g$ with integer coefficients
$n_g$, and the property that for each horoball $HB$ at $e$ the set
$\{g\in G\mid n_g\neq 0 \text{ and } gb\notin HB\}$ is finite. We
observe that $\widehat {{\Z}G^{e}}$ is a right $G$-module which contains
the group ring as a submodule, and we call it the {\it Novikov module}
at $e$. Then we have

\begin{thm} $e\in {^{\circ} \Sigma}^{n}(M;A)$ if and only if 
$\text{\rm Tor}_{k}^{{\Z}G}(\widehat {{\Z}G^{e'}},A)=0$ for all $0\leq k\leq n$
and all $e'$ contained in the closure of the orbit $Ge$ in $\del M$.
\end{thm}

\noindent This is useful since it opens the possibility of relating the various ${^{\circ} \Sigma}^{n}(M;A)$ via the long
exact Tor sequences. It indicates that $\Sigmacirc (M;A)$ is perhaps
better behaved that $\Sigma (M;A)$ with respect to the module argument.
The Polyhedrality Conjecture of \cite{BGe16} reinforces this view.

\item In Section \ref{Hn} we analyze the case where $M$ is Gromov hyperbolic (and proper $CAT(0)$). Among other things we prove that when $G$ acts properly discontinuously on $M$ and $\Sigmacirc (M;{\Z})\neq \emptyset$ then $G$ is of type $FP_{\infty}$.

\item An interpretation of the definitions of $\Sigmadoublecirc (M;A)$
and $\Sigmacirc (M;A)$ in terms of matrices\footnote{See Theorems
\ref{theta} and \ref{matrix}} shows that their complements in $\del M$
can be viewed as generalizations of the Bergman fan, which is defined
(for $M$ Euclidean and $G$ abelian) in \cite{Be71} and is proved to be
polyhedral in \cite{BGr84}. Indeed, the Invariance Theorem, asserting
that $\Sigmacirc (M;A)$ and $\Sigmadoublecirc (M;A)$ are independent
of the particular free presentation of $A$, shows that the matrix
interpretation is, in fact, a condition on the stable ${\Z}G$-matrices,
which suggests a $K$-theoretic connection.

\end{enumerate} 
\vskip 5pt
The point of this work is generality: We extend substantial parts of
known Euclidean techniques to the $CAT(0)$ case. That analyzing the
result in specific cases of $CAT(0)$ $G$-spaces --- Euclidean, hyperbolic
and mixtures thereof --- leads to intriguing concrete observations should
not be surprising.
\vskip 5pt
{\it Acknowledgment:} We thank Eric Swenson for helpful conversations concerning the Gromov-hyperbolic case discussed in detail in Section \ref{hyperbolic2}. 

\vskip 20pt

\section{Extended Outline}\label{S:E}

Because the general theory is somewhat involved, we give an outline here, leaving most of the technicalities for later sections.

\subsection{The finitary category of $G$-modules}\label{category}

Throughout the paper we will use the symbol $fS$ to denote the set of all finite subsets of a
given set $S$.     

Let $A$ and $B$ be $G$-modules. An additive homomorphism
$\varphi:A\to  B$ is $G$-{\it finitary} if it is captured by a $G$-map
$\Phi:A\to fB$, in the sense that $\varphi (a)\in \Phi (a)$ for every
$a\in A$. For brevity we say finitary rather than $G$-finitary if there is
no doubt which group action is under consideration. The $G$-map $\Phi$
is a {\it volley} for the finitary map $\varphi$, and $\varphi$ is a
{\it selection} of the volley $\Phi$.

Every $G$-homomorphism is, of course, $G$-finitary, but $G$-finitary
homomorphisms are much more general. Unlike a $G$-homomorphism, a
$G$-finitary map $\varphi:A\to B$ is not uniquely determined by its
values on a ${\mathbb Z}G$-generating set $X$ of $A$; however, the
possible values on $a = gx$ (where $g\in G$ and $x\in X)$ are restricted
to be in the finite set $\Phi(a) = g\Phi(x)\subseteq g\Phi(X)$. Finitary
homomorphisms are easy to construct when $A$ is the free $G$-module on
a set $X$: any $G$-finitary homomorphism $\varphi:A\to B$ can then be
given by first choosing  $\Phi(x)\in fB$ for each $x\in X$, and then
picking  $\varphi(gx)\in g\Phi(x)$ for all $(g,x)\in G\times X$.

\begin{example} When $A=B={\mathbb Z}G$ the $G$-finitary endomorphisms
$\varphi:{\mathbb Z}G\to {\mathbb Z}G$ with the special feature that
$\Phi(1)\in fG$ have an interpretation in terms of a ``semi-flow'' on
the Cayley graph\footnote {For an arbitrary subset $T\subseteq G$, the
Cayley graph $\Gamma (G,T)$ is the graph with vertex set $G$ and edge
set $G\times T$, where $g$ is the origin and $gt$ the terminus of the
edge $(g,t)$.} $\Gamma =\Gamma(G,\Phi(1))$: $\varphi$  can be regarded
as a map which selects for each vertex $g$ an edge with origin $g$ (and
hence terminus in $g\Phi (1)$).  \end{example}

In Section \ref{S:3.2a} we will observe that both sums and compositions of $G$-finitary maps
are
$G$-finitary, so that the class of all finitely generated $G$-modules and finitary maps is an
additive category -  the finitary category of $G$-modules. We do not know whether every
projective $G$-module is a projective object in the finitary category of $G$-modules, but the
following weaker property will do for our purposes.

\begin{lemma}\label{projective} Let $P$ be a projective $G$-module. For every
$G$-epimorphism 
$\alpha :A\twoheadrightarrow B$ and every $G$-finitary map $\varphi: P\to B$ there is a
$G$-finitary map ${\tilde \varphi } :P\to A$ such that $\alpha \circ {\tilde \varphi }=\varphi $.
\end{lemma} 

\begin{proof} It is enough to prove the lemma for the case when $P=F$ is a free $G$-module
over a basis $X\subseteq F$.  Let $\varphi$ be a selection from the volley $\Phi :F\to fB$. For
each $x\in X$ we find a finite set ${\widetilde \Phi}(x) \subseteq A$ with $\alpha {\widetilde \Phi}(x)=\Phi (x)$, and this defines
a canonical volley ${\widetilde \Phi}:F\to fA$. By $G$-equivariance we have 
$\alpha {\widetilde \Phi}(y)=\Phi (y)$ for every $y\in Y=GX$. Hence, for each $y\in Y$ we can
pick an element $c\in {\widetilde \Phi}(y)$ with $\alpha (c)=\varphi (y)$. A selection from  
$\widetilde \Phi$ is uniquely defined by its values on the $\mathbb Z$-basis $Y$, so we are done.
\end{proof}

\begin{examples}
\begin{enumerate}[1.]
\item If $H\leq G$ is a subgroup of finite index, and $A$, $B$ are $G$-modules then every
$H$-homomorphism $\varphi: A\to B$ is $G$-finitary.
\item If $N\leq G$ is a finite normal subgroup, and $A$ is a $G$-module then the additive
endomorphism of $A$ given by multiplication by 
$\lambda\in {\mathbb Z}N$ is $G$-finitary.
\end{enumerate}
\end{examples}

\begin{proof} For the first example, let $T\subseteq G$ be a transversal
for the right cosets $Ht$, and write $\overline {g}\in T$ for the representative
of $g\in G$ in $T$. We put $\Phi(a):=\{ t^{-1}\varphi(ta)\mid t\in
T\}$, noting that $\varphi(a)\in \Phi(a)$ and that $\Phi(a)$ is
independent of the particular choice of $T$. This allows us to infer
that, for all $g\in G, 
\Phi(ga) 
=\{t^{-1}\varphi (tga)\mid t\in T\}
= \{t^{-1}\varphi (tg(\overline {tg})^{-1}\overline {tg}a)\mid t\in T\} 
= g\{(\overline {tg})^{-1}\varphi (\overline {tg}a)\mid t\in T\} 
= g\Phi (a)$.

For the second example, $\Phi(a):= Na$ defines a volley for the endomorphism given by
multiplication by some $n\in N$; multiplication by $\lambda \in {\mathbb Z}N$ is a ${\mathbb
Z}N$-linear combination of such.
\end{proof}

\subsection{Control maps on finitely generated free $G$-modules}

Now the proper $CAT(0)$ space $M$ on which $G$ acts by isometries enters the picture. Relating the action of  $G$ on  $M$ to the action of $G$ on a finitely generated $G$-module $A$ starts with choosing a
``control map'' on a free presentation of $A$. Our free presentation of $A$ will always be given by a finite set $X$, the free G-module
$F = F_{X}$  over $X$, and an epimorphism $\epsilon : F\twoheadrightarrow
A$. The free $G$-set $Y=GX$ is a ${\mathbb Z}$-basis for $F_{X}$. The {\it support}
of an element $c\in F_{X}$, $\text {supp}(c) \subseteq Y$, is the set
of all elements $y\in Y$ occurring in the unique expansion of $c$ over
${\mathbb Z}$.

Recall that we write $fM$ for the $G$-set of all finite subsets of $M$. By a
{\it control map} on $F$ we mean a $G$-map $h: F\to fM$ given by composing
the support function $\text{supp}:F\to fY$ with an arbitrary $G$-equivariant map $fY\to fM$,
where $h(0)$ is defined to be the empty set. Thus $h$ is uniquely given by
its restriction $h|:X\to fM$. We will always assume that
our control maps $h$ are {\it centerless} in the sense that $h(x)$
is non-empty for all $x\in X$ (and hence $h(c)\neq \emptyset$ for all
$0\neq c\in F$).

\begin{lemma}[Bounded displacement]\label{bounded} Let $h: F\to fM$ and $h': F'\to fM$  be
control maps on $F$ and $F'$. For every $G$-finitary homomorphism $\varphi: F\to F'$
there is a number $\delta >0$ with the property that for each $c\in F$, the set $h'(\varphi(c))$ lies
in the
$\delta$-neighborhood of $h(c)$. 

\hfill$\square$
\end{lemma}

\subsection{Limit points of subsets of $F$ in $\partial M$}\label{limit}

The proper $CAT(0)$ space $M$ has a compact {\it
boundary at infinity} which we denote by $\partial M$. A point $e\in
\partial M$ is an equivalence class of proper rays $\gamma$ in $M$
where any two rays in the class lie within a bounded distance of one
another. The class contains exactly one ray starting at each point
of $M$.  In particular, given a ray
$\gamma$ there is a corresponding Busemann function\footnote{We refer to \cite {BrHa99} for
details on $CAT(0)$ spaces. However, we adopt the convention that the parameter $t$ in $\gamma (t)$ goes to $\infty$ as $t$ goes to $\infty$, the opposite convention is used in \cite {BrHa99}.} $\beta _{\gamma}:M\to {\R}$;.  

\begin{Rem}\label{gamma}{\rm While the value of $\beta _{\gamma}(c)$
depends on $\gamma$, the value of the difference $\beta _{\gamma}(c)-\beta
_{\gamma}(c')$ only depends on the equivalence class (i.e. the boundary
point) $e\in \partial M$.  In a context where a base point for $M$
has been chosen we permit ourselves the notation $\beta _{e}$, tacitly
assuming that the ray defining $e$ is the one which starts at the base
point. Similarly for horoballs at $e$: we will write 
$HB_{\gamma ,t}$ or $HB_{e,t}$ according to this convention to denote the unique horoball at $e$ with the
point $\gamma (t)$ on its frontier.} 
\end{Rem}

There are various definitions of what it means to say that $e\in \partial M$ is a limit point of a sequence in $M$ (or in $fM$). To complete the definition one must specify a filtration of $M$ which plays the role of a basis for the neighborhoods of $e$ through which the sequence converges to $e$. Two possibilities for this, leading to different kinds of limit points, are:
\begin{itemize}
\item The horoball-filtration of $M$ by the horoballs at $e$; this defines ``horospherical limit point."   
\item The cone-filtration of $M$ by the conical neighborhoods of $e$; this defines ``cone topology  (or Poincar\'{e}) limit point." 
\end{itemize}
We will refer to horoballs and cone-neighborhoods of $e$ as ``neighborhoods (of $e$ in $M$)" when we want to discuss these two kinds of limits at the same time. 

Referring, as usual, to a control map $h: F\to fM$, we say that $e$ is a {\it
limit point} of the subset $S\subseteq F$ if 
every neighborhood of $e$ contains $h(s)$ for some $s\in S$. Applying Lemma
\ref{bounded} to an automorphism of $F$ shows that for any two control maps $h$ and
$h'$ on $F$ there is a number $\delta>0$ with the property that $h(c)$ and $h'(c)$ are in
$\delta$-neighborhoods of one another. This shows that the concept of limit point is
independent of the choice of $h$. We write $L(S) \subseteq \partial M$ for  
the set of all limit points of $S$. Note that $L(S) =\partial M$ when $0\in S$.

For each free presentation $\epsilon : F\twoheadrightarrow A$ and each
$e\in\partial M$ we define
$$A_{e}(\epsilon ) : = \{a\in A\mid e\in L(\epsilon^{-1}(a))\}.$$
Thus $a\in A_{e}(\epsilon )$  means that for every neighborhood $N$ of $e$ there is some $c\in F$
such that $\epsilon (c)=a$ and $h(c)\subseteq N$.  We express this by saying that the element $a$ is {\it supported} over every neighborhood of $e$. 

\begin{lemma}\label{finitary} If $\varphi: A\to A'$ is a finitary
homomorphism of finitely generated $G$-modules, given with respective
finitely generated free presentations $\epsilon : F\twoheadrightarrow A$
and $\epsilon ': F' \twoheadrightarrow A'$, and endowed with control
functions, then $\varphi (A_{e}(\epsilon ))\subseteq A'_{e}(\epsilon
')$.  \end{lemma}

\begin{proof} By Lemma \ref{projective} there is a finitary homomorphism
$\psi : F\to F'$ with $\varphi\epsilon  = \epsilon  '\psi$, and by Lemma
\ref{bounded} there is a number $\delta \geq 0$ with the property that
for each $c\in F$, all of the set $h'(\psi (c))$ lies within $\delta$
of $h(c)$. Assume now that $a\in A_{e}(\epsilon )$, and let $N$ be
a neighborhood of $e$ in $M$. Then there is an element $c\in F$ with
$\epsilon (c) = a$ and $h(c)\subseteq N$. From the fact that $\epsilon '\psi
(c) = \varphi\epsilon (c) = \varphi(a)$, and the fact that $h'(\psi (c))$
lies within $\delta$ of $N$ we infer that $\varphi (a)\in A'_{e}(\epsilon
')$.  \end{proof}

Applying Lemma \ref{finitary} to $\varphi = \text {id}_{A}$ yields 
$A_{e}(\epsilon )=A_{e}(\epsilon ')$. Hence $A_{e}(\epsilon )$ -- or, equivalently, the limit set
$L(\epsilon ^{-1}(a))$ -- depends only on the $G$-module $A$ and the element $a\in A$, and
not
on the particular free presentation. Therefore from now on we will write $A_{e}$  for 
$A_{e}(\epsilon )$, and $L_{A}(a)$ for $L(\epsilon ^{-1}(a))$. We summarize by observing: 

\begin{thm}[Functoriality]\label{functoriality}  Let $e\in\partial M$. Then $(-)_{e}$ is a functor from
the finitary category of $G$-modules to the category of abelian groups.
Moreover, $A_{ge} = gA_{e}$ for all $g\in G$.
\hfill$\square$
\end{thm}

We remark that if $H\leq G$ is a subgroup of finite index then $A_e$ is
the same, whether $A$ is regarded as a $G$-module or as an $H$-module.

The {\it horospherical limit set} (resp.{\it cone topology  limit set})
over $M$ of the finitely generated $G$-module $A$ is
 $$\Sigma(M; A):= \{e\in \partial M \mid A_{e} = A\} \text { based on the
 horoball-filtration.}$$ $$\Lambda(M; A):= \{e\in \partial M \mid A_{e}
 = A\} \text { based on the cone-filtration.}$$

Thus
\begin{equation*}
\begin{split}
\Sigma (M;A)= \bigcap _{a\in A}L_{A}^{\text {\it horo}}(a)\\
\Lambda (M;A)= \bigcap _{a\in A}L_{A}^{\text {\it cone}}(a)
\end{split}
\end{equation*}

In other words: $\Sigma(M; A)$ (resp. $\Lambda (M; A)$) is the set of all
boundary points $e$ with the property that every element of the module
$A$ is supported over every horoball at $e$ (resp. every cone neighborhood
at $e$.)

\begin{Rem}\label{limitset}We will often use the common notation $\Lambda (G)$ for
the limit set $\Lambda (M;{\Z})$; it is the limit set of any orbit $Gb$,
$b\in M$. When $A\neq 0$ $ \Lambda (M;A)\subseteq \Lambda (G)$.
\end{Rem}

\subsection{Pushing a free module towards a boundary point}\label{push}

The Busemann function $\beta_{\gamma}$ maps $M$ to $\mathbb R$. We extend $\beta
_{\gamma}$ canonically to a map on the finite subsets of $M$, $\beta _{\gamma}: fM\to
{\mathbb R}_{\infty}:={\mathbb R}\cup \{\infty\}$,
by taking the minimum, with the convention $\beta_{\gamma}(\emptyset)=\infty$.

For a free $G$-module $F$, with specified finite basis $X$, and a control
function $h: F\to fM$ we can now consider the composition $v_{\gamma}:
=\beta_{\gamma}\circ h:F\to {\R}_{\infty}$ which we call the {\it
valuation on} $F$ {\it defined by} $h$ {\it and} $\gamma$. Usually a base point for $M$
will be understood, and then we will write $v_e$ rather than $v_{\gamma}$; compare Remark
\ref{gamma}.

The infimum of numbers $\delta $ such that, for all $c\in F$,
$v_{e}(\varphi(c))-v_{e}(c)\geq \delta $ is called the {\it guaranteed
shift towards }$e$ of $\varphi $. This number is denoted by $\text
{gsh}_{e}(\varphi )$. When $\text {gsh}_{e}(\varphi )>0$ we say that
$\varphi$ {\it pushes} $F$ {\it towards} $e$.

There is an insightful way to express the assertion $A_{e} = A$ in the above definitions: Consider an element $y$ of the ${\mathbb Z}$-basis $Y=GX$ of $F$. If $A_{e} = A$ then, given any $\delta>0$, one
can choose an element $\varphi(y)\in F$, representing the same element
$\epsilon (y) = \epsilon (\varphi (y))\in A$, such that 
$v_{e}(\varphi(y))- v_{e}(y)\geq \delta$. This choice defines an additive endomorphism
$\varphi :F\to F$ which lifts the identity map of $A$ and pushes $F$ towards $e\in\partial M$.  Conversely, the existence
of an additive endomorphism $\varphi: F\to F$ pushing $F$ towards $e$
and satisfying $\epsilon  = \epsilon \circ \varphi$ implies that $e$ is a
horospherical limit point of each coset of $F \text { mod ker}(\epsilon )$;
i.e., $A_{e} = A$. This is because, given $c\in F$ and a horoball $HB$ at $e$, there is some $n\in {\N}$ such that $\varphi ^{n}(c)$ is over $HB$, and  $\epsilon  = \epsilon \circ \varphi ^{n}$. 

\begin{equation}\label{sigmadef}
\Sigma (M;A) =\{e\in \partial M\mid \exists \varphi \,\in
\text{End}_{\Z}(F) \text { with } \epsilon \varphi = \epsilon \text {
and } \text {gsh}_{e}(\varphi)>0\}
\end{equation}

The definition of what we call the {\it equivariant-dynamical limit set} $\Sigmadoublecirc (M;A)$
contrasts neatly with this:

$$\Sigmadoublecirc (M;A): =\{e\in \partial M\mid \exists \, \varphi \in \text{End}_{{\Z}G}(F) \text { with } \epsilon \varphi = \epsilon \text { and } \text {gsh}_{e}(\varphi)>0\}$$

In between $\Sigma $ and $\Sigmadoublecirc $ is the $G$-finitary version,
the {\it finitary-dynamical limit set} $\Sigmacirc (M;A)$:

$$\Sigmacirc (M;A): =\{e\in \partial M\mid \exists \, G\text {-finitary }\varphi \in \text{End}_{\Z}(F) \text { with } \epsilon \varphi = \epsilon \text { and } \text {gsh}_{e}(\varphi)>0\}$$

\noindent In view of (\ref{sigmadef}) we see that $\Sigmadoublecirc (M;A)\subseteq \Sigmacirc (M;A)\subseteq \Sigma (M;A)$, the distinction being expressed by the kind of endomorphism which pushes $F$ towards $e$ while commuting with $\epsilon$. 

\begin{Rem}\label{fixedpoint}
\rm{When $e\in \del M$ is fixed under the $G$-action and $e\in \Sigma (M;A)$ we can do better: we can choose the values of $\varphi $ to satisfy the pushing-towards-$e$ inequalities on the ${\Z}G$-basis $X$, and then extend this to a $G$-endomorphism on $F$.
Thus, the three invariants coincide when restricted to points of $\del M$ fixed by $G$.}
\end{Rem}

\subsection{The main results}

\subsubsection{\bf Cone topology properties}

The invariants $\Sigma (M;A), \Sigmacirc (M;A)$ and $\Sigmadoublecirc
(M;A)$ are well-defined, independent of all choices, and are invariant
under the action of $G$ on $\del M$, i.e. if they contain $e$ they
contain the whole orbit $Ge$. But more is true.

\begin{thm}\label{closure3} Both  $\Sigmacirc (M;A)$ and $\Sigmadoublecirc (M;A)$ contain the closure of $Ge$ whenever they contain $e$.
\end{thm}

\begin{thm}\label{closure} $\Sigmacirc  (M; A) = \{e\in\partial M\mid \text
{\rm cl}(Ge) \subseteq \Sigma(M; A)\}$.
\end{thm}

\begin{cor}\label{invariant} For each closed $G$-invariant subset $E\subseteq \partial M$ we have $E\subseteq  \Sigma(M; A)$ if and only if $E\subseteq \Sigmacirc  (M; A)$. In particular, $\Sigma(M; A) =\partial M$ if and only if $\Sigmacirc  (M; A) =\partial M$. 
\end{cor}

\subsubsection{{\bf When} $\Sigma (M;A)=\del M$}

We say that $A$ {\it has bounded support over} $M$ if there is a bounded subset $B\subseteq M$ with the property that each element $a\in A$ is represented by an element $c\in F$ over $B$.

\begin{thm}\label{all} Let the finitely generated $G$-module $A$ be non-zero. Then $\Sigma(M; A) =\partial M$ if
and only if $G$ acts cocompactly on $M$ and $A$ has bounded support over $M$.
\end{thm}

A fuller version of this is given as Theorem \ref{uniform}. 

Bounded support over $M$ is not an intrinsic property of a $G$-module,
as it also involves the metric of $M$. However, when the $G$-action on
$M$ has discrete orbits then $A$  has bounded support over $M$
if and only if $A$ is finitely generated as a module over the stabilizer
$G_b$ of some (equivalently, any) point $b\in M$. More precisely (see Corollary \ref{smaller}):

\begin{cor}\label{total} Let $b\in M$, let the $G$-orbits in $M$
be discrete, and let the module $A$ be non-zero. Then $\Sigma (M;A)
=\partial M$ if and only if the $G$-action on $M$ is cocompact and $A$
is finitely generated as a $G_b$-module.  
\end{cor}

We will see that when $\Sigmacirc  (M; A) =\partial M$ then $G$-finitary
endomorphisms of $F$ which push $F$ towards the various boundary points
$e\in\partial M$ can all be obtained as selections from a single volley
$\Phi: F\to fF$. That a volley $\Phi$ has selections pushing towards all
of $\partial M$ can be expressed in terms of finitely many inequalities,
and these inequalities remain fulfilled when the action $\rho : G\to
\text{\rm Isom}(M)$ is subject to small perturbation. This is the idea which
leads to the following openness theorem (proved by combining Corollaries
\ref{open2} and \ref{kernel}, below):

\begin{thm} Let ${\mathcal R}(G,M)$ denote the space of all
isometric actions $\rho :G\to \text{\rm Isom }(M)$ which have discrete orbits,
endowed  with the compact-open topology. Then for
every finitely generated $G$-module $A$ and point $b\in M$ the subset 
$$\{\rho \in  {\mathcal R}(G,M)\mid A \text{ is finitely generated over } G_b\}$$ is open
in  ${\mathcal R}(G,M)$.
\end{thm}

\subsubsection{\bf Tits metric properties}

Let  $\epsilon : F\twoheadrightarrow A$ be a finitely generated free
presentation of $A$. Each of
$\Sigma (M;A)$, $\Sigmacirc (M;A)$ and $\Sigmadoublecirc (M;A)$ can
be described as the union of subsets of the form 
$$\Sigma (\varphi ):=\{e\mid \text{gsh}_{e}(\varphi)>0\}$$ 
where $\varphi $ runs through
all endomorphisms of $F$ of the appropriate kind (additive, $G$-finitary or $G$-equivariant) satisfying $\epsilon \circ \varphi =\epsilon$. Even
though $\Sigma (\varphi )$ is not invariant under changes of  control
map or presentation $\epsilon $, it is the key to the our Tits
metric results. 

In Section \ref{pushes} we study $\Sigma (\varphi )$ in relation to a set 
$\Lambda (\varphi )$ which is a subset of the cone topology  limit set $\Lambda
(M; {\Z})$:
$$\Lambda (\varphi):=\{e\mid \exists c\in F \text { such that } e \text{
is a limit point of a sequence } (y_{k})\\ \text { with } y_{k}\in
\text{supp}(\varphi ^{k}(c))\}$$

\noindent and we prove, among other things:

\begin{thm}\begin{enumerate}[(a)]
\item If the module $A$ is non-zero and $\Sigma (\varphi )$ is non-empty then $\Lambda (\varphi )$ non-empty.  
\item For every pair $(e,e')\in \Sigma (\varphi)\times \Lambda (\varphi )$ the angular distance $d(e, e')$ is at most $\frac{\pi}{2}$.
\item If $\varphi $ is $G$-finitary then $\Sigma (\varphi)\cup \Lambda (\varphi)$ lies in a Tits metric ball of radius $<\frac{\pi}{2}$
\item If $\varphi $ is $G$-finitary then $\Sigma (\varphi )$ is an open subset of $\del M$ in the Tits metric topology.  
\end{enumerate}
\end{thm}
\begin{proof} $G$-finitary maps have finite norm; hence the assertions follow from the results in Section \ref{pushes} which are established under this weaker assumption.
\end{proof}

\begin{thm}\label{open1}For every finitely generated $G$-module $A$ the subsets $\Sigmacirc (M;A)$ and $\Sigmadoublecirc(M;A)$ of $\del M$ are open in the Tits metric topology.
\end{thm}

\begin{proof}Let $\epsilon :F\twoheadrightarrow A$ be a finitely generated
free presentation of $A$, and let a control map be chosen. The norm
of a $G$-finitary map is always finite, so  $\Sigmacirc (M;A)$ is the
union of sets $\Sigma (\varphi )$ where $\varphi $ runs through all
$G$-finitary endomorphisms of $F$ which commute with $\epsilon $ and
satisfy $\text {gsh}_{e}>0$ for some $e\in \Sigmacirc (M;A)$. And when
the union is restricted to those $\varphi $ which are $G$-endomorphisms
we get $\Sigmadoublecirc(M;A)$. Openness therefore follows from Theorem
\ref{Titsopen}.  \end{proof}

Theorem \ref{near} and Corollary \ref{tighter} imply:
\begin{thm} With respect to the Tits metric we have, for all finitely generated non-zero $G$-modules $A$, 
\begin{enumerate}[(i)]
\item $\Sigma (M;A)$ lies in the closed $\frac{\pi}{2}$-neighborhood of $\Lambda (M;{\Z})$.
\item For some $r<\frac{\pi}{2}$, $\Sigmacirc (M;A)$ lies in the closed $r$-neighborhood of $\Lambda (M;{\Z})$.
\end{enumerate}
\hfill$\square$
\end{thm}

\begin{thm}\label{lang} When $e\in \Sigmadoublecirc (M;A)$ there exists
$r(e)<\frac{\pi}{2}$ such that $e$ lies in the $r(e)$-neighborhood of
a point of $\del M$ that is fixed by $G$.  
\end{thm}

\begin{proof}Let $\epsilon :F\twoheadrightarrow A$ be a finitely generated
free presentation of $A$, and let a control map be chosen. If $e\in
\Sigmadoublecirc (M;A)$ then $F$ admits a $G$-endomorphism $\varphi $
with $\text{gsh}_{e}(\varphi )>0$. Lemma \ref{L:3.2} implies that for each
$g\in G$ the $g$-translate $g\varphi $ pushes all of $F$ towards $ge$
with $\text{gsh}_{ge}(\varphi)=\text{gsh}_{e}(\varphi )$. But\footnote{ The group $G$ acts {\it diagonally} on the set Hom$_{\mathbb Z}(A,B)$ of ${\mathbb Z}$-homomorphisms; this means that when $g\in G$ and $\varphi : A\to B$ is a ${\mathbb Z}$-homomorphism,  $g\varphi : A\to B$ is defined by $(g\varphi)(a) = g\varphi(g^{-1}a)$.} $g\varphi =\varphi$. It follows that both $\Sigma (\varphi )$ and $\Lambda (\varphi )$ are $G$-invariant. By Corollary \ref{tighter}, $\Sigma (\varphi )\cup \Lambda (\varphi )$ lies in a ball of radius $<\frac{\pi}{2}$. By Theorem
B of \cite{LS97} every subset lying in a ball of radius $<\frac{\pi}{2}$
lies in a unique minimal circumball with unique center. In this case, the set
is $G$-invariant, so the center is fixed by $G$.  
\end{proof}

Theorem \ref{double} follows from Theorem \ref{closure3} and Theorem \ref{lang}.

\subsection{$\Sigma (M;A)$ in various contexts}\label{translation}

\subsubsection{\bf Euclidean translation action}\label{subsection}

Here we assume that $M={\E}^m$ is a finite-dimensional Euclidean space
and $G$ acts by translations via $\rho :G\to \text{Transl}({\E}^m)$. The
convex hull of the orbit $Gb$ is a subspace ${\E}^{n}\subseteq {\E}^m$
which contains all limit sets; i.e. $\Lambda ({\E}^{n};A) = \Lambda
({\E}^{m};A)$ and $\Sigma ({\E}^{n};A)=\Sigma ({\E}^{m};A)\cap \del{\E}^n$, 
so we can restrict attention to the cocompact space ${\E}^{n}$. The orbit $Gb$
may or may not be discrete in $\E^{n}$.  The horoballs of ${\E}^{n}$
are half spaces, and $\partial {\E}^{n}$ is the sphere at infinity
$S^{n-1}$. The induced action of $G$ on $S^{n-1}$ is trivial. By Corollary
\ref{invariant} it follows that for every finitely generated $G$-module
$A$ we have

$$\Sigma({\E}^{n};A)=\Sigmacirc ({\E}^{n};A)=\Sigmadoublecirc ({\E}^{n};A)$$

In the case when ${\E}^n$ is cocompact with discrete orbits it can be
viewed as ${\E}^{n}=\rho (G)\otimes {\R}$ equipped with an inner product,
with $G$ acting on $\rho (G)$ by left multiplication. When $\rho (G)$
is the abelianization of $G$ we call this the {\it canonical Euclidean
$G$-space}.  In that case we recover a special case of the ``Geometric
Invariant" of \cite{BNS87}, $\Sigma (G_{\text{ab}}\otimes {\R};A)$ which
in \cite{BRe88} is defined as 
$$\Sigma ^{0} (G;A):=\{e\mid A\text{ is finitely generated as a }G_e\text{-module}\}.$$
In Theorem \ref{fg} we show that this agrees with $\Sigma (G_\text {ab}\otimes {\R};A)$; i.e. 
$\Sigma (G_\text {ab}\otimes {\R};A) \text { equals } \Sigma ^{0} (G;A)$.

Since $\Sigma ^{0} (G;A)$ is the model case for the role of the module argument
in applications, a short review of its precise relationship with the
geometric invariants of  \cite{BNS87} and \cite{BRe88} is in order.

\vskip 5pt
{\bf Digression: Review of the BNSR Invariants}
\vskip 3pt
{\bf A. The} $\Sigma ${\bf -invariants of} \cite{BRe88}\label{Sigma}

We assume that the abelianization $G_{\text{ab}}$ is of  finite $\Z$-rank so that
${\E}^{n}=G_{\text{ab}}\otimes {\R}$ is finite-dimensional. To each
$e\in \partial {\E}^{n}$ is associated the homomorphism $\chi _{e}:G\to
{\R}$ given by the inner product  with the unit vector of ${\E}^{n}$
in the direction $e$, and the submonoid of $G$, $G_{e}=\{g\mid \chi
_{e}(g)\geq 0\}$. 

The (homological) geometric invariants $\Sigma ^{k}(G;A)$ of \cite{BRe88}
are open subsets of $\del (G_{\text{ab}}\otimes {\R})$. They are defined
when the $G$-module $A$ is of type $FP_k$ as:

\begin{equation}\label{our2}
\Sigma ^{k}(G;A)=\{e\mid A \text{ is of type } FP_{k} \text{ as a module over the monoid ring } {\Z}G_{e}\}
\end{equation}

The corresponding homotopical invariants $\Sigma ^{k}(G)$, introduced
in \cite{BRe88} and investigated in \cite{Re89}, are also open subsets
of $\del (G_{\text{ab}}\otimes {\R})$. They are defined for $k\geq 0$
when the group $G$ admits a cocompact free action on a $(k-1)$-connected
$CW$-complex $X$ as follows: on the free $G$-$CW$-complex $X$ choose (as
we always can) a continuous $G$-map $h:X\to {\R}$ which ``extends" $\chi
_{e}:G\to {\R}$ in the sense that $h(gx)=\chi _{e}(g)+h(x)$ for all $g\in
G$ and $x\in X$. 

$$ \Sigma ^{k}(G)=\{e\mid X \text{ and } h \text{ can be chosen so that }
h^{-1}([0,\infty ) \text{ is }(k-1)\text{-connected}\}$$

\noindent By the Hurewicz Theorem $\Sigma ^{k}(G)\subseteq \Sigma ^{k}(G;{\Z})$
 when both invariants are defined, and $\Sigma ^{1}(G)=\Sigma
 ^{1}(G;{\Z})$ for all finitely generated groups.  Bestvina and Brady
 \cite{BB97} provide finitely presented groups $G$ where
$\Sigma ^{2}(G)$ is a proper subset of $\Sigma ^{2}(G;{\Z})$.

{\bf B. The older} $\Sigma ${\bf -invariants of} \cite{BNS87}\label{older}

For $G$ abelian, $\Sigma ^{0}(G;A)$ was originally introduced in [BS80].
The noticeable similarity between the openness of $\Sigma (G;A)$ and
W. Neumann's openness result \cite{Ne79} for arbitrary finitely generated
groups eventually led to \cite{BNS87} which contains as its major tool
the invariant $\Sigma _{N}$ defined as follows for any finitely generated
group $G$ and any finitely generated $G$-group $N$:

$$\Sigma _{N}:=\{e\mid N \text{ is finitely generated as a } P\text{-group for some finitely generated submonoid }P\subseteq G_{e}\}.$$


\begin{rem} (A hidden sign-problem.) The actions of the group $G$ on the $\Sigma $-invariants are
sensitive to whether $G$ acts on the left or on the right
of spaces and resolutions. Thus comparing $\Sigma $-invariants
in different publications might require a sign which sends each point
of the boundary sphere to its antipode. For example, since \cite{BS80} and \cite{BNS87}
follow the convention that conjugation is a right action, the original
$BNS$-invariant of \cite{BNS87} would be antipodal to ours (see Formula $(1.3)$
in \cite{BRe88}). 
\end{rem}


The relationship\footnote{The simpler subset 
$\Sigma '_{N}:=\{e\mid N \text{ is finitely generated as a }G_{e}\text{-group}\}$ turned out to be less powerful for applications.} between $\Sigma _{N}$ and $\Sigma ^{k}(G;A)$ includes two notable features:
\begin {itemize}
\item If $N=A$ is a finitely generated $G$-module then ${\Sigma } _{A}=\Sigma ^{0}(G;A)$;
\item If $N=G'$ is the commutator subgroup of $G$ then $\Sigma _{G'}=\Sigma ^{1}(G;{\Z})$, where the action of $G$ on $\Z$ is trivial; see \cite{BRe88}. 
\end{itemize}

Thus the invariant $\Sigma ^{0} (G;A)$ (which equals $\Sigma
(G_\text {ab}\otimes {\R};A)$ of the present paper) has extensions in
two directions: The Bieri-Neumann-Strebel extension which replaces $A$
by a non-abelian $G$-group $N$, and the Bieri-Renz extension to higher
dimensions. The two extensions have substantial intersection beyond
$\Sigma ^{0} (G;A)$: this intersection contains the invariant $$\Sigma
_{G'}=\Sigma ^{1}(G)=\Sigma ^{1}(G;{\Z})$$ which plays a crucial role
in the theory and is therefore often referred to as the $\Sigma$-{\it
invariant} (or the {\it Bieri-Neumann-Strebel invariant}) of $G$. When
$G$ is a $3$-manifold group $\Sigma _{G'}$ recovers the Thurston norm
\cite{Th86}.

\begin{Rem} 
\begin{enumerate}[1.]
\item If $M$ is a proper $CAT(0)$ $G$-space and $A$ is a
finitely generated $G$-module we do have corresponding extensions of
our horospherical limit set $\Sigma (M;A)$, namely higher dimensional
invariants  $\Sigma ^{k}(M;A)$ (the present case being $k=0$) defined when
$A$ is of type $FP_k$, and even $\Sigma (M;N)$ where $A$ is replaced by a non-abelian $G$-group $N$. These will appear in
subsequent papers.
\item Higher-dimensional homotopical invariants $\Sigma ^{k}(M)$
have already been investigated in \cite{BGe03}, hence 
$\Sigma ^{1}(M)=\Sigma ^{1}(M;{\Z})$ is already available.  
\end{enumerate}
\end{Rem}

{\bf C. The case when the group $G$ is abelian}\label{abelian2}

The case when the group $G$ is abelian goes back to the paper \cite{BS80} where $\Sigma ^{0}(G;A)$ was introduced as a tool to decide exactly when a finitely generated metabelian group $\Gamma$ which fits into a short exact sequence 
$1\to A\to \Gamma \to G\to 1$ admits a finite presentation:

\begin{thm}\label{Strebel} $\Gamma $ is finitely presented if and only if $\Sigma ^{0} (G;A)$ together with its antipodal set covers $\del (G\otimes {\R})$.
\end{thm}
\begin{rem} If $A$ is not a $G$-module but merely a $G$-group containing no non-abelian free subgroups, $\del (G\otimes {\R})= -\Sigma (G,A_{\rm ab})\cup \Sigma (G,A_{\rm ab})$ is still a necessary condition for finite presentability of $\Gamma $.
\end{rem}

The set $\Sigma ^{0} (G;A)$ also determines  whether $\Gamma$ is of type $FP_{\infty}$
(for metabelian groups this is equivalent to the existence of a $K(\Gamma
,1)$-complex with finite skeleta). The conjunction of results in \cite{BS82} and \cite{BGr82} yields:	
 \begin{thm}\label{finite complement} $\Gamma $ is of type $FP_{\infty}$ if and only if
the complement of $\Sigma ^{0} (G;A)$ is finite and is contained in an
open hemisphere.
\end{thm}

In fact there is rather precise but well justified

\begin{conj}(The $FP_m$-Conjecture for metabelian groups) $\Gamma $ is of type $FP_n$ if and only if  every $n$-point subset of $\Sigma ^{0} (G;A)^c$ lies in an open hemisphere of 
$\del (G\otimes {\R})$.  
\end{conj} 

This conjecture appeared in print in \cite{BGr82}. Aberg \cite{A86}, Noskov
\cite{No97}, Kochloukova \cite{K96}, Bux \cite{Bu97}, and Bieri-Harlander
\cite{BHa01} have contributed further results towards its verification, but the
general case remains open.

{\bf D. Connection to tropical algebraic geometry}\label{short tropical}

Such applications are not the only point of interest: Also interesting is the mathematics developed in the effort to compute $\Sigma ^{0} (G;A)$ explicitly when $G={\Z}^n$. In that case $\Sigma ^{0} (G;A)$ depends only on the annihilator ideal of $A$ in the Laurent
polynomial ring ${\Z}G$, $I=\text{Ann}_{{\Z}G}(A)$. The main result of
\cite{BGr84} exhibits the complement of $\Sigma ^{0} (G;A)$ in $\del (G\otimes
{\R})$ as the radial projection of a certain rational-polyhedral
subset $\Delta ^{{\Z}}\subseteq G\otimes {\R}={\E}^n$, i.e.
$$\Sigma ^{0} (G;A)^{c}=\Sigma (G;{\Z}G \slash I)^{c}=\del({\R}_{>0}\Delta
^{{\Z}}).$$
$\Delta ^{{\Z}}$ is defined in terms of valuations
on the commutative ring ${\Z}G \slash I$. Some
fifteen years later it turned out, see \cite {EKL06}, that $\Delta
^{{\Z}}$ is the integral version of what is now called the {\it tropical
variety} associated to the ideal $I$ (or the 
{\it tropicalization}\footnote{For an up-to-date introduction to tropical
geometry with
a certain emphasis on computational aspects see \cite{MS15}.} of the
algebraic variety $V$ of $I$). In \cite{BGr84} $\Delta ^{D}$ was
investigated over a Dedekind domain $D$ in order to include fields
as well as ${\Z}$. The field version of \cite{BGr84} anticipated some
fundamental facts at the roots of tropical geometry, 
among other things the result that if $V$ is irreducible then its
tropicalization is of
pure dimension equal to $\text{dim}V$.
 
For more details see the Appendix.

\subsubsection{\bf The case when $M$ is  Gromov-hyperbolic}\label{hyperbolic}

The proper $CAT(0)$ space $M$ is {\it Gromov-hyperbolic} if for some
$\delta \geq 0$, every geodesic triangle in $M$ lies in the $\delta
$-neighborhood of any two of its sides. We write $\widehat M$ for $M\cup
\partial M$; this is a compact metrizable space, and the given action
of $G$ on $M$ extends to an action on $\widehat M$ by homeomorphisms. To
avoid trivialities we assume the $G$-orbits in $M$ are unbounded, and the given
finitely generated $G$-module $A$ is non-zero. As before, we write $\Lambda (G)$
for the cone topology  limit set of $G$ in $\del M$; i.e. $\Lambda (G)=\Lambda
(M;{\Z})$.

By an {\it interval} in $\widehat M$ we mean any one of: a closed geodesic segment in
$M$, a geodesic ray in $M$ together with its end point in $\partial M$,
or a line in $M$ together with its two end points in $\partial M$. When
$S\subseteq \widehat M$ we define $S[1]=S$ and inductively for $n\geq 2$ $S[n]$ is the union of all intervals whose endpoints lie in $S[n-1]$. If $S$ is $G$-invariant, so is $S[n]$. In the literature, $S[2]\cap M$ is sometimes called the {\it weak convex hull} of $S$.

Among the results in Section \ref{hyperbolic2} are the following: 

\begin{thm}\label{alternatives} $\Sigmacirc (M;A)$ is either empty, or is a singleton set, or coincides with the limit set $\Lambda (M;A)$.  Moreover, the following are equivalent:
\begin{enumerate}[(i)]
\item $\Sigmacirc (M;A)$ contains at least $2$ points;
\item $\Sigmacirc (M;A)=\Sigma (M;A)=\Lambda (M;A)=\Lambda (G)\neq \emptyset $; 
\item $\Lambda (G)[2]$ is cocompact, and $A$ has bounded support over $\Lambda (G)[2]$.
\end{enumerate}
\end{thm}

When the $G$-orbits in $M$ are discrete, the phrase ``$A$ has bounded support over $M$" becomes ``$A$ is finitely generated over the point stabilizer $G_b$". Hence:

\begin{cor}\label{finitely generated} Assume the $G$-orbits in $M$ are discrete. Then $\Sigma (M;A)=\Lambda (G)$ if and only if $\Lambda (G)[2]$ is cocompact and $A$
is finitely generated over the stabilizer $G_b$ of $b$.
\end{cor}

 We note that, unlike the similar-sounding Corollary
\ref{total}, Corollary \ref{alternatives} gives a characterization of
``$\Sigma (M;A)=\Lambda (M;A)$" in situations where $M$ itself need not
be cocompact. This is also something to note about the next theorem:

\begin{thm}\label{finite type} Assume the $G$-orbits in $M$ are discrete,
and that the stabilizer $G_b$ of some (any) point $b$ is finite. If
$\Sigmacirc (M;{\Z})$ is non-empty then $G$ is of type $F_{\infty}$.
\end{thm}

\subsubsection{\bf The case when $M$ is the hyperbolic space ${\H}^n$}

Specializing to the case where $M={\H}^n$ and $G=\Gamma $ is an
infinite discrete subgroup of $\text{Isom}({\H}^n)$, we can relate these
results to standard properties of discrete hyperbolic groups. We replace
$\Lambda (G)[2]$ by the convex hull of the limit set. It is shown in Section \ref{Hn} that if $\Sigmacirc ({\H}^{n};\Z)$ is non-empty then $\Gamma $ is geometrically finite. Geometrically finite groups are well understood; see \cite{Bo93} for the definition. The limit set of
such a group $\Gamma $ is the disjoint union of its conical limit points
and its parabolic fixed points (see Section \ref{Hn}
for definitions and details). From this we deduce:

\begin{prop}\label{conical} If $\Gamma $ is geometrically finite then $\Sigma (_{\Gamma }{\H}^{n};\Z)$ is the set of its conical limit points.
\end{prop}

On the basis of Theorem \ref{alternatives} we have:

\begin{cor} $\Sigmacirc ({\H}^{n}; {\Z})$ is
non-empty if and only if $\Gamma $ is geometrically finite and has no
parabolic fixed points.
\end{cor}

\subsubsection{\bf The case when $M$ is a symmetric space.}\label{Rehn1}

The symmetric space $M:=\text{ SL}_{n}({\mathbb R})/\text{SO}(n)$
is a contractible $d={\frac{1}{2}}(n-1)(n+2)$-dimensional Riemannian
manifold of non-positive curvature, hence it is a proper $CAT(0)$ space,
and $\partial M$ is the sphere $S^{d-1}$. This $M$ lies between the extremes of the previous two subsections as it has both higher-dimensional flats and higher-dimensional hyperbolic complete geodesic subspaces.

The sphere-boundary also carries the structure of the spherical building
associated to $\text {SL}_{n}({\mathbb R})$, whose apartments are
$(n-2)$-spheres represented by the points at infinity of the maximal tori
of $\text{ SL}_{n}({\mathbb R})$. We call such an apartment ``rational"
if its torus is defined over ${\mathbb Q}$, and we write $B_{\mathbb
Q}$ for the union of all rational apartments in $\partial M$. Thus
$B_{\mathbb Q}$ is a subset of $\partial M$ which can be viewed as a
geometric realization of the spherical building associated to $\text{
SL}_{n}({\mathbb Q})$. Avramidi and Witte-Morris \cite {AWM} have recently
proved a theorem which settles a conjecture of Rehn \cite{Reh09} ---
a conjecture which was open for a number of years:

\begin{thm}For $G=\text {\rm SL}_{n}({\mathbb Z})$, acting on the symmetric
space $M=\text{\rm SL}_{n}({\mathbb R})/\text{\rm SO}(n)$, the horospherical
limit set $\Sigma (_{G}M;{\Z})$ is the complement of $B_{{\Q}}$ in
$\del M$.  \end{thm}

\noindent Those authors have a more general result which characterizes
the horospherical limit set whenever $M$ is the universal cover of a
finite volume locally connected symmetric space $M\slash\Gamma $ of
non-compact type; see \cite {AWM}.


\section{Controlled free $G$-modules}\label{S:2}

In this section we provide details on control maps $h:F_{X}\to
fM$ defined on a based free $G$-module. These are tools to keep the effect of
endomorphisms of $F_X$ geometrically under control. In later sections
we will apply this to free presentations of a $G$-module $A$.

\subsection{The support function}\label{S:2.1}

By the {\it support of an element} $c\in F_X$ we mean the set of all
elements of $Y$ (=$GX$) occurring with non-zero coefficient in the unique
expansion of $c$:
\begin{equation*} \text{ supp}_Y\left(\sum_{y\in Y}
n_yy\right) := \{y\in Y\mid n_y \neq 0\}.  
\end{equation*}
Thus the
support is a function $\text { supp} : F_X \to fY$.  A special case
is the support function on the group algebra, $\text { supp} : {\mathbb Z}G \to fG$,
where the $X$ is the singleton basis $\{1\}$ and hence $Y = G$.

\subsection{Control maps on free modules}\label{S:2.2}

Consider a $G$-map $h : F_{X}\to fM$ defined as follows:  Starting with
an arbitrary choice of non-empty $h(x) \in fM$, for each $x\in X$,
we extend this to a $G$-map $h : Y\to fM$ on the ${\mathbb Z}$-basis $Y=GX$, and
for any non-zero element $c\in F$ we put $h(c) = h(\text{supp}(c))$.
Define $h(0)=\eset $. The map $h : F_X \to fM$ defined in this way is a {\it control map}. 

A control map $h$ satisfies:

\begin{enumerate}[(i)]
\item $h(c)=\eset $ if and only if $c=0$,
\item $h(rc) = h(c)$, for all $r\neq 0\in {\mathbb Z}$ and $c\in F$, and 
\item $h(c+c') \subeq h(c) \cup h(c')$ for all $c,c'\in F$.
\end{enumerate}

In the context of a given a base point $b\in M$, if we define each $h(x) =
\{b\}$ we call $h$ the {\it canonical control map} on $F_X$ corresponding
to $b\in M$. A {\it controlled based free $G$-module} is a based free $G$-module
equipped with a control map to $fM$.

\subsection{Valuations on free modules}\label{S:2.3}

Let the point $e\in \partial M$ be determined by the geodesic ray $\gamma : [0,\infty) \to M$.
Composition of the control map $h : F \to fM$, with the Busemann function $\beta_{\gamma} :
M\to {\mathbb R}$ assigns to each element of $F$ a finite set of real numbers; taking minima
defines the function  

\begin{equation}\label{E:2.2} 
v_\gamma := \min \beta_\gamma h : F \to {\mathbb R}_{\infty }.
\end{equation}

In particular $v_{\gamma}(c)=\infty$ if and only if  $c=0$. Generalizing \cite{BRe88} we call $v_\gamma$ a {\it valuation} on $F$.

\begin{lemma}\label{L:2.3}
\begin{enumerate}[{\rm (i)}]
\item $v_\gamma(-c) = v_\gamma(c)$,
\item $v_\gamma(c+c') \geq \min\{v_\gamma(c),v_\gamma(c')\}$,
\item $v_\gamma(c) = v_{g\gamma }(gc)$, for all $g\in G$.
\item If $c$ and $c'$ are non-zero then
$v_\gamma(c)-v_\gamma(c')$ depends only on the endpoint $\gamma(\infty)
= e$, not on the ray $\gamma$, and $|v_\gamma(c)-v_\gamma(c')| \leq d_{H}(h(c),h(c'))$, where
$d_H$ denotes Hausdorff distance.
\end{enumerate}
\hfill$\square$
\end{lemma}

\subsection{Controlling homomorphisms over $M$}\label{S:3.3}

Let the based free modules $F_{X}$ and $F'_{X'}$ be endowed with control maps $h$ and $h'$
mapping to $M$.   We want to measure how far, in terms of the metric $d$ on $M$, a ${\mathbb
Z}$-homomorphism ${\varphi} : F \to F'$ moves
the members of $F$.  We define the {\it norm} of $\varphi $ by 
\begin{equation}\label{E:3.4} 
||\varphi|| :=\inf\{r\geq 0\mid h'(\varphi (c))\subseteq N_{r}(h(c)) \text{ for all }c\in F\}\in
{\mathbb R}\cup \{\infty\}
\end{equation} 
{\it the shift function towards} $e$, $\text{sh}_{\varphi,e}
: F\to {\mathbb R} \cup \{\infty\}$, by 
\begin{equation}\label{E:3.5}
\text{sh}_{\varphi,e}(c) := v'_\gamma(\varphi(c)) - v_\gamma(c) \in {\mathbb
R}\cup \{\infty\},\ c\in F, 
\end{equation} 
and the {\it guaranteed shift towards} $e$ by, 
\begin{equation}\label{E:3.6}
\text{gsh}_e(\varphi) := \inf\{\text{sh}_{\varphi,e}(c)\mid c\in F\}.
\end{equation}

The next two lemmas collect properties of norm and shift, the relations between them, and their behavior with respect to compositions and the $G$-action. By convention, $G$ acts {\it diagonally} on the set $\text {\rm Hom}(A,B)$ of all additive homomorphisms between $G$-modules $A$ and $B$, namely, for $g\in G$ and $\varphi \in \text {\rm Hom}(A,B)$, $(g\varphi)(a)=g\varphi (g^{-1}a)$.

We call a ${\mathbb Z}$-submodule $L\leq F_X$ {\it cellular} if it is generated by
$L\cap Y$.  While the most important case is $L=F_X$, sometimes another $L$ will be given,
and we will be interested in the norm or guaranteed shift of $\varphi |L$.  To have
information for that case we include $L$ in these lemmas.

\begin{lemma}\label{L:3.2} Let $\varphi :L\to F'$ be the restriction to $L$ of a 
${\mathbb Z}$-homomorphism $F\to F'$.
\begin{enumerate}[{\rm (i)}]
\item  $\text{\rm sh}_{\varphi,e}(y) \geq -||\varphi||$ for all $y\in L\cap Y$; hence {\rm
gsh}$_e(\varphi) \geq -||\varphi||$.
\item $||g\varphi|| = ||\varphi||$, {\rm sh}$_{g\varphi,ge}(gc) = \text{\rm sh}_{\varphi,e}(c)$ for all $c$, and {\rm
gsh}$_{ge}(g\varphi) =
\text{\rm gsh}_e(\varphi)$, for all $g\in G$
\end{enumerate}
\hfill$\square$
\end{lemma}

\begin{lemma}\label{L:3.3}
Let $\varphi : F\to F'$ and $\psi : F'\to F''$ be two ${\mathbb Z}$-endomorphisms, and let $K
\leq F$ and $L\leq F'$ be cellular ${\mathbb Z}$-submodules with
$\psi(K)\subeq L$.  Then 
\begin{equation*}
\text{\rm gsh}_e(\varphi |L\circ \psi |K) \geq \text{\rm gsh}_e(\varphi |K) + \text{\rm gsh}_e(\psi
|L).
\end{equation*}
In particular, 
\begin{equation*}
\text{\rm gsh}_e(\varphi^k)\geq k\cdot \text{\rm gsh}_e(\varphi), \text{ for all natural numbers
}k.
\end{equation*}
\end{lemma}

\begin{proof} We use Lemma \ref{L:3.2}(ii).
\begin{equation*}
\begin{aligned}
\text{gsh}_e((\varphi |)\circ (\psi |K)) &= \inf_{c\in K}\{v''_\gamma(\varphi\psi(c)) -
v_\gamma(c)\}\\
&= \inf_{c\in K}\{v''_\gamma(\varphi \psi(c)) - v'_\gamma(\psi(c)) + v'_\gamma(\psi(c)) -
v_\gamma(c)\}\\
&\geq \inf_{c\in K}\{v''_\gamma(\varphi \psi(c)) - v'_\gamma(\psi(c))\} 
+ \inf_{c\in K}\{v'_\gamma(\psi(c)) -v_\gamma(c)\}\\
&\geq \inf_{b\in L}\{v''_\gamma(\varphi(b)) -v'_\gamma(b)) + \inf_{c\in
K}(v'_\gamma(\psi(c))-v_\gamma(c)\}\\
&= \text{gsh}_e(\varphi |L) + \text{gsh}_e(\psi |K).
\end{aligned}
\end{equation*}
\end{proof}

We say that the ${\mathbb Z}$-endomorphism $\varphi : F\to F$ {\it pushes $L$ towards} $e\in
\partial M$, and we call $\varphi $ a {\it push towards} $e$, if the guaranteed shift of $\varphi
|L$ towards $e$ is positive; i.e., gsh$_e(\varphi|L) > 0$.

 8pt
\subsection{Pushes and limits}\label{pushes}

Let $\varphi :F\to F$ be an additive endomorphism of the controlled based free $G$-module $F=F_X$. A ${\Z}$-basis for $F$ is $Y=GX$. We consider two limit sets of sequences of elements of $F$ (over $M$) in the sense
of Section \ref{limit}:
$$\Sigma (\varphi):=\{e\in \del M\mid \text{gsh}_{e}(\varphi)>0\}$$
and

$$\Lambda (\varphi):=\{e\in \del M \mid \exists c\in F 
\text { such that } e \text{ is a limit point of a sequence } (y_{k})\\
\text { with } y_{k}\in \text{supp}(\varphi ^{k}(c))\}$$

\begin{prop}\label{not empty}Assume there is an element $c\in F$
such that $\varphi ^{k}(c)\neq 0$ for all $k$. If $\Sigma (\varphi )$ is
non-empty then so is $\Lambda (\varphi )$.  
\end{prop} 

\begin{proof} 
By definition, the sequences $(\varphi ^{k}(c))$ with $c\in F$
horo-converge to each $e\in \Sigma (\varphi)$. Picking a $c$ satisfying
the non-zero assumption, we can choose $y_{k}\in \text{supp}(\varphi
^{k}(c))$. The accumulation points of $(h(y_{k}))$ cannot be in $M$ so
they must be in $\Lambda (\varphi )$.  
\end{proof}

\begin {thm}\label{near} For each pair $(e,e')\in \Sigma (\varphi)\times \Lambda
(\varphi )$ we have $d(e, e')\leq 
r=\text{\rm arccos}(\frac{\text {\rm gsh}_{e}(\varphi )}{||\varphi ||})$ where $d$ is the Tits metric.
\end{thm}

\begin{rems}\begin{enumerate}[1.]
\item Since $0<\text{gsh}_{e}(\varphi )\leq ||\varphi ||$ this shows that the distance $d(e, e')$ is, in general, at most $\frac{\pi }{2}$, and $r<\frac{\pi }{2}$ when $||\varphi ||<\infty $.

\item Theorem \ref{near} shows that when $||\varphi ||<\infty $ both
$\Sigma (\varphi)$ and $\Lambda (\varphi)$ have diameter $<\pi $. Hence we
can infer from the $CAT(1)$ property of $\del M$ that $\Sigma (\varphi)$
and $\Lambda (\varphi)$ have well-defined convex hulls $\widehat {\Sigma
(\varphi)}$ and $\widehat {\Lambda (\varphi)}$. But then the $CAT(1)$
property shows that the assertion of Theorem \ref{near} holds when $e\in
\widehat {\Sigma (\varphi)}$ and $e'\in \widehat {\Lambda (\varphi)}$. Thus
$\widehat {\Lambda (\varphi)}$ lies in a closed circumball of radius
$r$. Being convex, it will contain the center $z$ of a circumball of
minimal radius. It follows that the ball of radius $r$ with center $z$
contains both $\widehat {\Sigma (\varphi)}$ and $\widehat {\Lambda
(\varphi)}$. Thus we have:

\end{enumerate}
\end{rems}
\begin{cor}\label{tighter} When $||\varphi ||<\infty $, $\Sigma (\varphi)\cup \Lambda (\varphi)$ lies in a Tits-metric ball of radius $r=\text{\rm arccos}(\frac{\text {\rm gsh}_{e}(\varphi )}{||\varphi ||})<\frac{\pi}{2}$.
\end{cor}

\begin{proof} (of Theorem \ref{near}) Let $(e,e')\in \Sigma (\varphi)\times \Lambda
(\varphi )$. Then $e'$ is an accumulation point of a sequence
$(p_{k})$ in $M$ such that, for some $c\in F$ and all $k\geq 0, $ 
$p_{k}\in h(\text{supp}\varphi ^{k}(c))$; and the sequence $(p_{k})$ horo-converges to $e$. 

We may assume $p_{0}=b$, the base point, since change of base point does not
affect guaranteed shift. We represent $e$ by a geodesic ray $\gamma $
emanating from $b$. 

Let $\alpha (k)$ be the angle at $b$ between $p_k$ and $\gamma $. Write 
$\delta = \text{gsh}_{e}(\varphi )$. By
Lemma \ref{L:3.3} we know that $p_k$ lies in $HB_{e, k\delta }$, the
horoball at $e$ with apex $\gamma (k\delta )$, so that for each $k$ and
all sufficiently large $t$ we have $t-k\delta \geq d(p_{k},\gamma (t))$.

Consider the geodesic triangles $\Delta (p_{k}, b, \gamma (t))$ and
their Euclidean comparison triangles $\Delta ^*$. We write $\alpha
(k,t)$ for the angle in $\Delta $ at $b$. By the $CAT(0)$ inequality,
the corresponding angle $\alpha ^{*}(k,t)$ is an upper bound for $\alpha
(k,t)$.  The two sides of $\Delta (p_{k}, b, \gamma (t))$ adjacent
to $b$ are of length $d(b,\gamma (t))=t$ and $d(b, p_{k})=:u$, while
the third side is of length $\leq t-k\delta $. Thus the Law of Cosines
gives 
$$u^{2}+t^{2}-2ut\text{cos}\alpha ^{*}(k,t)\leq (t-k\delta)^{2},$$
and hence, in the limit, $\text{cos}\alpha ^{*}(k,\infty)\geq
\frac{k\delta }{u}$. By continuity of angles at $b$ (\cite{BrHa99} 9.2(1)) this implies $\alpha
(k,\infty)=\text {lim}_{k\to \infty}\alpha (k,t)$ is the angle at
$b$ of the ideal geodesic triangle $\Delta (p_{k},b,e)$. Thus we have
$\text{cos}\alpha (k,\infty)\geq \frac{k\delta}{u}$.

When the norm $||\varphi ||$ is finite we know that $u\leq k||\varphi
||$, hence $\text{cos}\alpha (k,\infty)\geq \frac{\delta}{||\varphi
||}$, a positive number independent of $k$. A subsequence $(p_{k_i})$
of $(p_{k})$ converges to $e'$, and, again by continuity of angles at
$b$, $\text {lim}_{i\to \infty}\alpha (k_{i},\infty )$ is the angle at $b$
between $e$ and $e'$. As we have found an upper bound for $\alpha (k,t)$
independent of $b$, this is also an upper bound for the angular (or Tits
metric) distance $d(e, e')$.

When the norm of $\varphi $ is infinite, we see that in the Euclidean triangle $\Delta ^*$ the side opposite $b^*$ is no longer than the side opposite $p_k$. Thus $\alpha ^{*}(k,t)$, the angle of $\Delta ^*$ at $b$, cannot be the largest of the three angles of $\Delta ^*$ and is therefore smaller than a right angle. Hence $\alpha (k,t)\leq \alpha ^{*}(k,t)<\frac{\pi}{2}$. The previous limit argument applies and we find $d(e, e')\leq \frac{\pi}{2}$.
\end{proof}

\subsection{$\Sigma (\varphi )$ is open in the Tits metric}

The following lemma gives information about how the value of the Busemann
function $\beta _{e}$ changes as $e$ varies over a Tits-metric
neighborhood:

\begin{lemma}\label{neighborhood} Let $r>0$ and $\epsilon >0$ be given,
and let $R\geq r(1+\frac{2r}{\epsilon})$. When geodesic rays  $\gamma $
and $\gamma '$ start at the same point $w$ and represent $e$ and $e'$
in $\del M$, and when $p\in B_{r}(w)$ then
\begin{equation*}
|\beta_{\gamma }(p) - \beta_{\gamma '}(p) | \leq \epsilon +2R \text{\rm sin}\frac{\angle(e,e')}{2})
\end{equation*}
\end{lemma}

\begin{proof}Lemma II8.21(1) of \cite{BrHa99} asserts that if $w\in M$, $p\in B_{r}(w)$, $u\notin B_{R}(w)$, and $v$ is the point on $[b,u]$ distant $R$ from $w$, then
$$0\leq d(p,v)+d(v,u)-d(p,u)\leq \epsilon .$$ 
Applying this to $u=\gamma (t)$, $v=\gamma (R)$ and $t>R$ gives
$$0\leq d(p,\gamma(R))+(t-R)-d(p,\gamma (t))\leq \epsilon $$
Letting $t\to \infty$ this gives
\begin{equation}\label{star}
0\leq d(p,\gamma(R))-R+\beta_{\gamma }(p)\leq \epsilon
\end{equation}
Considering (\ref{star}) for both $\gamma$ and $\gamma'$, and taking the difference of the two inequalities we get
\begin{equation}\label{doublestar}
|\beta_{\gamma }(p) - \beta_{\gamma '}(p) | \leq \epsilon + d(\gamma (R),\gamma '(R))
\end{equation}
By Proposition III.3.4 of \cite{BrHa99}, the sequence 
$\frac{d(\gamma (R),\gamma '(R)}{R}$ with $R\to \infty$ is non-decreasing and its limit is 
$2\text{sin}\frac{\angle(e,e')}{2}$. By (\ref{doublestar}) the Lemma follows.
\end{proof} 

\begin{Rem}\label{core}
If we choose  $e'\in \del M$ so that $|\angle (e,e')| <2\text{arcsin}\frac{\epsilon}{2R}$ we find 
$|\beta_{\gamma }(p) - \beta_{\gamma '}(p) | \leq 3\epsilon $. Thus if $\beta_{\gamma }(p)>\delta >0$ and we put $\epsilon =\frac{\delta}{9}$, then (with $r$ and $R$ as above) we find a Tits-metric neighborhood $N$ of $e$ such that when $e'\in N$ $|\beta_{\gamma }(p) - \beta_{\gamma '}(p) | \leq \frac{\delta}{3}$, and therefore $\beta_{\gamma '}(p)\geq \frac{\delta}{3}$.
\end{Rem}

\begin{thm}\label{Titsopen}Let $F$ be a finitely generated free $G$-module, and let $\varphi :F\to F$ be a ${\Z}$-map with finite norm. Then $\Sigma (\varphi )$ (defined with respect to
an arbitrary control map $h$) is an open subset of $\del M$ in the Tits
metric topology.  \end{thm}

\begin{proof} Let $e\in \Sigma (\varphi)$ and let $\text{gsh}_{e}(\varphi )\geq \delta >0$. We write $r:= ||\varphi ||<\infty $. Let $p\in h(x)$ realize $\beta _{e}(h(x))$ in the sense that $\beta _{e}(p)$ is minimal among points of $h(x)$, and let $q\in h(\varphi (x))$. There is some $w\in h(x)$ such that $q\in B_{r}(w)$. Let $\gamma $ and $\gamma' $ be rays starting at $w$ and defining $e$ and $e'$ respectively. The limitations on $e'$ will be determined later.

By Remark \ref{core} there is a Tits neighborhood $N_1$ of $e$ such that when $e'\in N_1$ we have
(1)\;\;\;\;\;\;$|\beta_{\gamma }(q) - \beta_{\gamma '}(q) | \leq \frac{\delta}{3}$.

Because of the guaranteed shift we have\\
(2)\;\;\;\;\;\;$|\beta_{\gamma }(q) - \beta_{\gamma }(p) | \geq \delta$.

And because $\beta _{e}(p)$ is continuous in $e$ with respect to the cone topology (hence also the Tits topology) there is a Tits neighborhood $N_2$ of $e$ such that when $e'\in N_2$ we have\\
(3)\;\;\;\;\;\;$|\beta_{\gamma }(p) - \beta_{\gamma '}(p) | \leq \frac{\delta}{6}$.

From these we get\\
(4)\;\;\;\;\;\;$|\beta_{\gamma '}(q) - \beta_{\gamma '}(h(x))| \geq  \frac{\delta}{6}$ when $e'\in N_{1}\cap N_{2}$.

Since $X$ is finite and each $h(\varphi (x))$ is finite, we may assume this holds for all $x$ and all $q\in h(\varphi (x))$.

It follows that this remains true if $x$ is replaced by any $y=gx$, and hence also if $x$ is replaced by any $c\in F$. 
\end{proof}

\section{$G$-finitary homomorphisms}\label{Sec4}

\subsection{$G$-volleys}\label{S:3.2}

Let $S$ and $T$ be $G$-sets. A  $G$-{\it volley from $S$ to $T$} is a
$G$-equivariant map $\Phi :S\to fT$. Two volleys $\Phi :S\to fT$ and $\Psi
:T\to fU$ can be can be ``composed'' to give the volley $\Psi\Phi : S\to
fU$ defined by $\Psi\Phi(s) := \ds{\bigcup_{t\in \Phi(s)}}\Psi(t)$. A
$G$-map $\varphi : S\to T$ may be regarded as the $G$-volley which
assigns to every element $s\in S$ the singleton set $\{\varphi(s)\}$.
Hence $G$-volleys and $G$-homomorphisms can be composed in the above
sense.

\begin{example} {\rm In this paper a $G$-volley will usually be given on based free
$G$-module $F_X$.  Indeed, if $B$ is an arbitrary $G$-module, every map
$\Phi :X\to fB$ extends to a {\it canonical volley} $\Phi
:F_X\to fB$ as follows:  On elements $y = gx$ of the ${\mathbb Z}$-basis $Y =
GX$, $\Phi$ is uniquely determined by $G$-equivariance:  $\Phi (y) :=
g\Phi(x)$; and for arbitrary elements $c = \ds{\sum_{y\in Y}}r_{y}y\in
F_X$, in the unique expansion, we put}
\begin{equation*}
\Phi (c) := \ds{\sum_{y\in Y}}r_{y}\Phi (y):=\{\ds{\sum_{y\in Y}}r_{y}b_{y}\mid b_{y}\in
\Phi (y),\text{ for all }y\in Y\}.
\end{equation*}
{\rm It is straightforward to check that $\Phi (gc)=g\Phi (c)$. We call 
$\Phi : F_X \to fB$ the {\it canonical $G$-volley induced by} $\Phi :X\to fB$.}  
\end{example}

\subsection{Finitary homomorphisms}\label{S:3.2a}

A {\it selection} from the $G$-volley $\Phi :A\to fB$ is a
${\mathbb Z}$-homomorphism $\varphi : A\to B$ such that $\varphi(a) \in \Phi(a)$ for all $a\in
A$. If $\varphi : A\to B$ is a selection from  the $G$-volley $\Phi$, so are all its (diagonal) 
$G$-translates $g\varphi : A\to B$. A ${\mathbb Z}$-homomorphism
$\varphi : A\to B$ is $G$-{\it finitary} if it is a selection from  some finite $G$-volley $\Phi :
A\to fB$. We note that an additive map $\varphi :A\to B$ is $G$-finitary if and only if 
$\{g\varphi(g^{-1}a)\mid g\in G\}$ is finite for each $a\in A$.

\begin{lemma}\label{L:3.1} 
If $\varphi : A\to B$ and $\psi : B \to C$ are $G$-finitary, so is the composition $\psi\varphi :
A\to C$.
\end{lemma}

\begin{proof}If $\varphi : A \to B$ and $\psi
: B\to C$ are selections from the volleys $\Phi : A\to fB$, $\Psi : B\to
fC$, respectively, then $\psi \varphi : A\to C$ is a selection from 
the composed volley $\Psi \Phi : A\to fC$.  
\end{proof}

Thus there is a $G$-finitary category of $G$-modules.

One observes readily that if $\varphi :F_{X}\to F'_{X' }$ is
a $G$-finitary homomorphism between two based free $G$-modules
endowed with control maps $h:F_{X }\to M$ and $h':F'_{X'}\to M$ then
its norm $||\varphi ||$ is finite. This is one of the coarse features
that $G$-finitary maps share with $G$-equivariant maps. For us this
is crucial: it implies the Bounded Displacement Lemma \ref{bounded},
and it implies that the guaranteed shift,  $\text {gsh}_{e}(\varphi)$,
is a well-defined real number for each $e\in \del M$.

Another important feature of finitary maps is the closure property (Theorem \ref{P:3.4}) as we now explain.

\subsection{Pushing submodules towards limit points of orbits in $\partial M$}\label{S:3.4}

We assume here that the cellular submodule  $L\leq F$
is in fact a $G$-submodule.  It will then be generated, as a
$G$-module, by $X' = L\cap X\subeq X$.  From Lemma \ref{L:3.2}(ii)
we know that if $\varphi$ pushes the $G$-submodule $L$ towards $e$ with guaranteed
shift $\delta$, then the $G$-translate $g\varphi$ of $\varphi$
pushes $L$ with the same guaranteed shift $\delta$ towards $ge$.
In the special case when $\varphi |L$ is $G$-finitary we can do better:
given any $\hat e\in\text{cl}(Ge)$, the closure of the $G$-orbit $Ge\subseteq \partial M$,
we can still construct $G$-finitary endomorphisms pushing towards $\hat e$:

\begin{thm}\label{P:3.4}({\rm Closure})  Let $L\leq F$ be a cellular $G$-submodule of
$F=F_X$ and let $\varphi : L \to F$ be a selection from  the
$G$-volley $\Phi :L\to fF$ with $\text{\rm gsh}_e(\varphi) = \delta
> 0$.  Then for every $\hat e\in \text{\rm cl}(Ge)$ there is a
selection $\psi : L\to F$ from $\Phi$ with $\text{\rm gsh}_{\hat e}(\psi)
\geq \frac {\delta}{2}$.  In fact, this can be done so that, on each finitely
generated ${\mathbb Z}$-submodule of $L$, $\psi$ coincides with  some
$G$-translate $g\varphi$.  
\end{thm}

\begin{proof}  Let $\hat e \in \partial M$ and let $(g_ke)_{k\in {\mathbb
N}}$ be a sequence of points in the orbit $Ge$ converging to $\hat e$.
The $G$-module $L$ is freely generated as a $\Z$-module by $Y'=GX'$.
We will define a map $\psi :Y'\to F$ with $\text{\rm gsh}_{\hat e}(\psi)
\geq \frac{\delta }{2}$, such that for all $y\in Y'$ $\psi (y)\in \Phi (y)$.
The extension of this to a ${\Z}$-map on $L$ will be the required map.

For $y\in Y'$ we have $\beta _{\hat e}((g_{k}\varphi)(y))-\beta _{\hat e}(y)=a+b+c$ where 

\begin{align*}
a &=\beta _{\hat e}((g_{k}\varphi)(y))-\beta _{g_{k}e}((g_{k}\varphi)(y))\\
b &=\beta _{g_{k}e}((g_{k}\varphi)(y))-\beta _{g_{k}e}(y)\\
c &=\beta _{g_{k}e}(y)-\beta _{\hat e}(y)
\end{align*}

We have seen that $b\geq \delta$, and there exists $N$ (dependent on $y$)
such that each of $a$ and $c$ is $<\frac{\delta }{4}$ when $k\geq N$;
this is clear for $a$, and holds for $c$ because $(g_{k}\varphi)(y)$
lies in the finite set $\Phi (y)$ for all $k$.  The required map $\psi $
is therefore defined by  $\psi (y)=(g_{N}\varphi)(y)$.
\end{proof}

\begin{Rem}\label{equivariant1} If $\varphi$ is a $G$-map, the volley $\Phi$ is just $\{\varphi\}$ and $g_{k}\varphi=\varphi$ for all $k$. Thus this special case is covered by Theorem \ref{P:3.4}.			
\end{Rem}

\section{The Dynamical Limit Sets $\Sigmacirc (M;A)$ and $\Sigmadoublecirc
(M;A)$}\label{S:5}

In this section we apply the work of Sections \ref{S:2} and \ref{Sec4} to a finitely generated based free
presentation $F\thra A$ of the $G$-module $A$.

The {\it dynamical limit sets} of the pair $(M,A)$ are defined to
be

\begin{equation*} \Sigmacirc(M;A):= \{e\mid \exists
\text{ a }G\text{-finitary endomorphism }\varphi :F\to F \text{
inducing id}_{A} \text{ with gsh}_{e}(\varphi )>0\} 
\end{equation*} 
 
and

\begin{equation*} \Sigmadoublecirc(M;A):= \{e\mid \exists 
\text{ a } G\text{-endomorphism }\varphi :F\to F \text{ inducing id}_{A}
\text{ with gsh}_{e}(\varphi )>0\} 
\end{equation*}

\begin{prop} [Invariance] \label{P:5.4} Let $e\in\partial M$. The
existence of a $G$-finitary endomorphism (resp. $G$-endomorphism)
$\varphi :F\to F$ inducing {\rm id}$_A$ and pushing $F$ towards
$e$ is independent of the choice of presentation $F\thra A$ and of
the control map $h:F\to fM$. In other words, $\Sigmacirc(M;A)$ and
$\Sigmadoublecirc(M;A)$ are well defined.  \end{prop}

\begin{proof} Let $F'\thra A$ be a second such presentation.  The identity
map id$_A$ can be lifted to $G$-maps $\alpha :F\to F'$, $\beta :F'
\to F$.  Assume there exists a $G$-finitary (resp. $G$-equivariant) push $\varphi : F \to F$
towards $e$ inducing id$_A$.  Then $\alpha\varphi\beta :F' \to F'$ is a $G$
finitary (resp. $G$-equivariant) map inducing id$_A$.  By Lemmas \ref{L:3.2} and \ref{L:3.3},
gsh$_{e}(\alpha\varphi^k\beta) \geq - ||\alpha|| + k\cdot \text {
gsh}_{e}(\varphi ^k) - ||\beta||$. The norm of a $G$-map is finite, so if
we choose $k$ large enough to ensure that $k\cdot \text { gsh}_{e}(\varphi
) > ||\alpha|| + ||\beta||$, the map $\alpha \varphi^k\beta :F'\to
F'$ becomes a $G$-finitary (resp. $G$-equivariant) push towards $e$ inducing $\text{id}_A$.
This shows independence of the free presentation.  Independence of
the control map is proved as a special case: take $F=F'$, $\alpha $
an automorphism, and $\beta $ the inverse of $\alpha $.  \end{proof}

We can now prove Theorem \ref{closure3}, the statement that the $G$-sets
$\Sigmacirc(M;A)$ and $\Sigmadoublecirc(M;A)$ contain the closure of
each of their orbits.

\begin{proof}(of Theorem \ref{closure3})  Let $\varphi :F\to F$ be a $G$-finitary push of $F$
towards $e\in \Sigmacirc (M;A)$ which induces $\text{id}_A$.  The proof
of Theorem \ref{P:3.4} constructs a $G$-finitary map $\psi :F\to
F$ pushing towards an arbitrary point of the closure of $Ge$ with
the property that for every ${\Z}$-basis element $y$, there is some
element $g\in G$, with $\psi (y)= (g\varphi)(y)$. Thus $\psi$ induces
$\text{id}_A$ as required.  The claim for $\Sigmadoublecirc(M;A)$ holds
because the volley defined by a $G$-endomorphism is a singleton; see Remark \ref{equivariant1}. 

\end{proof}

\begin{prop}\label{P:5.6} $e\in \Sigmacirc (M;A)$ if and only if for
each $x\in X$, there is a finite subset $\Phi(x) \subseteq F$, with
$\epsilon(\Phi(x)) = \epsilon(x)$ such that each of the functions $\mu_x
: \partial M\to {\mathbb R}$, $\mu_x(e') := \max\{v_{\gamma '}(\Phi(x))
- v_{\gamma '}(x)\}$ satisfies one of the following equivalent properties:
\begin{enumerate}[(i)]
\item $\mu _{x}$ is positive on $\text{\rm cl}(Ge)$ 
\item $\mu _{x}$ has a positive lower bound on $Ge$.
\end{enumerate}
\end{prop}

\begin{proof}  The two conditions on $\mu_x$ are equivalent since $\mu_x$ is continuous.  Let
$\delta :=\inf
\mu_x(Ge)$.  Let $\Phi :X\to fF$ be given, and extend $\Phi$ to a canonical finite $G$-volley
$F_{X}\to fF$. Because $\delta >0$, a selection $\varphi : F\to F$ from $\Phi$ pushing $F$
towards $e$ and inducing ${\rm id}_A$ can be defined as follows.  For each $y= g^{-1}x \in
GX$ pick an element $c(g,x) \in\Phi(x)$
with $v_{g\gamma}(c(g,x)) - v_{g\gamma}(x)\geq \delta$, and put $\varphi(y) := g^{-1}c(g,x)
\in g^{-1}\Phi(x) =\Phi(y)$.  By Lemma \ref{L:2.3}(iii), $v_\gamma(\varphi(y)) - v_\gamma(y) =
v_\gamma(g^{-1}c(g,x))-v_\gamma(g^{-1}x) = v_{g\gamma}(c(g,x))-v_{g\gamma}(x)\geq
\delta$.  The converse follows immediately from the definition of $\Sigmacirc (M;A)$.
\end{proof}

\subsection{$\Sigmacirc (M;A)$ in terms of matrices over ${\mathbb Z}G$}\label{matrices}

We specialize Proposition \ref{P:5.6} by making the following choices:

\begin{enumerate}[(1)]
\item $F=({\Z}G)^n$ and $X$ is the canonical basis;
\item A base point $b$ is chosen in $M$, and the control map $h$ maps the canonical basis
to the singleton set $\{b\}$;
\item for each $e\in \partial M$ we write $v_e$ for the canonical valuation taking each basis
element $x$ to $0$. 
\end{enumerate}

Let ${\M}_{n}({\Z}G)$ denote the ring of $n\times n$ matrices with
entries in ${\Z}G$. The information contained in the volley $\Phi$ can
also be expressed by the finite set of matrices $\Theta ^{+}\subseteq
{\M}_{n}({\Z}G)$ describing the restrictions $\varphi \mid X$ to the
chosen basis of all the selections $\varphi$ from $\Phi$. For each $g\in
G$ the selection $\varphi $ chooses one of the matrices in $g\Theta ^{+}$
to exhibit the restriction $\varphi \mid gX$. The condition $\epsilon (\varphi
(x))=\epsilon (x)$ for all $x\in X$ becomes the statement that each of the
matrices $\theta ^{+}\in \Theta ^{+}$ yields a map $\theta ^{+}:A^{n}\to
A^{n}$ which fixes the generating family ${\bf a}=(a_{1},\dots , a_{n})$
of $A$, where $a_{i}=\epsilon (x_{i})$. In other words, $\theta ^{+}{\bf
a}={\bf a}$ for all  $\theta ^{+}\in \Theta ^{+}$.

For any $\eta \in {\M}_{n}({\Z}G)$ we write $v_{e}(\eta)$ for the minimum
value of $v_{e}$ on the entries of $\eta$. This measures the shift towards $e$
of the map ${\Z}^{n}\to ({\Z}G)^n$ given by ${\bf u}\mapsto \eta {\bf
u}$. (Here, ${\Z}^{n}={\Z}X\subseteq ({\Z}G)X=({\Z}G)^n$.) The matrix
version of Proposition \ref{P:5.6} now reads:

\begin{thm}\label{theta} Let ${\bf a}\in A^n$ be a generating set for
the $G$-module $A$. Then $e\in \Sigmacirc (M;A)$ if and only if there
is a finite subset $\Theta^{+}\subseteq {\M}_{n}({\Z}G)$ of matrices
$\theta ^+$ satisfying   $\theta ^{+}{\bf a}={\bf a}$ such that the
following two equivalent conditions hold: 
\begin{enumerate}[(i)]
\item For each
$e'\in \text{\rm cl}(Ge)$ there is some $\theta ^{+
}\in \Theta ^{+}$ such that
$v_{e'}(\theta ^{+})>0$; 
\item There exists $\epsilon >0$ such that for
each $g\in G$ some $\theta ^{+} \in \Theta ^{+}$ satisfies $v_{ge}(\theta
^{+})\geq \epsilon$.  
\end{enumerate} 
\hfill$\square$
\end{thm}

It is sometimes more convenient to use the matrix $\theta :=1-\theta
^{+}$ rather than $\theta ^{+}$ (note that $\theta ^{+}{\bf a}={\bf a}$
if and only if $\theta {\bf a}={\bf 0}$), together with the following
notion of a ``minimal part with respect to $e$": Each $e\in \partial M$
gives rise to an $\R$-grading of the additive group ${\Z}G$ as follows:
each $\lambda \in {\Z}G$ has a canonical sum decomposition $\lambda =
{\Sigma }_{r\in \R}\lambda _r$, where $\lambda _r$, the homogeneous
component of degree $r$, collects all the monomials $n_{g}g$ ($g\in G,
n_g\in \Z$) with $v_{e}(g) = r$. Note that $\lambda _{s}\neq 0$ for
only finitely many $s\in \R$. If $\lambda \neq 0$ then $v_{e}(\lambda
)\in \R$, and $\lambda _{v_{e}(\lambda)}$ is called the {\it initial
term} of $\lambda$ with respect to $e$; it is denoted by $\lambda
_e$. Thus $\lambda = \lambda _{e}+\lambda ^{+}$ with $v_{e}(\lambda
_{e})=v_{e}(\lambda)$ and $v_{e}(\lambda ^{+})>v_{e}(\lambda )$. If
$\lambda =0$ we set $\lambda _{e}:=0$. This extends to matrices as
follows: The ${\Z}G$-matrix $\eta $ also has an ${\R}$-grading where the $r$th
grade is the matrix consisting of the $r$th grade of each entry. Define
$\eta _e$ to be the least-indexed non-zero grade of $\eta$, and define
$\eta ^{+}$ by $\eta =\eta _{e}+\eta ^{+}$.

When $e\in \Sigmadoublecirc (M;A)$ Theorem \ref{theta} holds with
$\Theta =\{\theta\}$, a singleton.  Hence we have:

\begin{thm}\label{matrix} Let ${\bf a}\in A^n$ be a generating set for
the $G$-module $A$.  Then $e\in \Sigmadoublecirc (M;A)$ if and only if
there is a matrix $\theta \in {\M}_{n}({\Z}G)$ such that $\theta {\bf
a}={\bf 0}$ and $\theta _{e}= {\bf 1}_{n}$.  
\end{thm}

\subsection{The case of $G$ abelian}\label{determinant}

In the special case when $G$ is torsion free abelian and $Ge=e$, Theorem 
\ref{matrix} can be considerably simplified by use of the determinant. One
can multiply the equation $\theta {\bf a}={\bf 0}$ on the left by the
cofactor matrix  $\theta ^{\text{cof}}$ of $\theta$, and this leads
to 

$$(\text{det}\theta )1{\bf a}=\theta ^{\text{cof}}\theta {\bf
a}={\bf 0}.$$

\noindent hence $(\text{ det}\theta )a_{i}=0$ for all $i$,
i.e. $\theta $ annihilates $A$.

Now, $\theta _{e}=1$ means that all entries of the matrix $\theta
^{+}:=1-\theta $ have positive value under $v_e$; and since $Ge=e$,
$v_{e}:{\Z}G\to {\R}_{\infty }$ satisfies $v_{e}(\lambda \lambda
')=v_{e}(\lambda)+v_{e}(\lambda ')$ for all $\lambda , \lambda '
\in {\Z}G$. It follows that $\text{ det}\theta = \text{ det}(1-\theta
^{+})$, which is of the form ${1}+$ non-empty products of entries of
$\theta ^{+}$, has initial term $(\text{det}\theta)_{e}=1$. Hence
the scalar matrix $(\text{ det}\theta )1$ has the same properties as
the matrix $\theta$ in Theorem \ref{matrix}. This proves

\begin{cor}\label{initial} When $G$ is torsion free and abelian,  $e\in
\Sigmadoublecirc (M;A)$ if and only if there is an element $\lambda
\in {\Z}G$ with $\lambda A=0$ and $\lambda _{e}=1$. In particular,
$\Sigmadoublecirc (M;A)$ is determined by the annihilator ideal
$I=\text{\rm Ann}_{{\Z}G}(A)$ of $A$ in ${\Z}G$.

\end{cor}

\begin{Rem} Corollary \ref{initial} thus leads to a concept at the roots
of tropical geometry: When $G$ is free abelian of finite rank $n$ and
$M=G\otimes {\R}$ then $\del M$ is the sphere $S^{n-1}$, and the set of
all directions $e$ with the property that the ideal contains an element
$\lambda$ with $\lambda_{e}=1$ is the complement of the Bergman fan of the
ideal $I$; see Appendix. By \cite{BGr84} we know that the
Bergman fan is polyhedral, i.e. a finite union of finite intersections
of hemispheres. It would be very interesting to find a generalization
of this polyhedrality to the non-positively curved context of Corollary
\ref{initial}. 
 \end{Rem}

\section{The Horospherical Limit Set $\Sigma(M ;A)$}\label{S:4}

The definition of the horospherical limit set 
\begin{equation} 
\Sigma (M;A):=\{e\in \partial M\mid A \text{ is supported over every horoball } HB\subseteq M \text { at } e\} 
\end{equation}

\noindent was given in Section \ref{intro}. Spelled out in detail this reads:

\begin{multline}\label{E:4.1}e\in \Sigma (M;A) \text{ if and only if for every } t\in {\mathbb R} \\  
\text{ and every }a\in A\text{ there is some }c\in F \text{ with }\epsilon(c) 
= a\text{ and }v_\gamma(c)\geq t. 
\end{multline}

That $\Sigma (M;A)$ is independent of choice of presentation and control map was proved in Section \ref{limit}. Here we collect some elementary facts about $\Sigma (M;A)$ related to the $G$-module argument $A$.

\begin{lemma} Let $A'\ {\overset {\phi} \to}\ A\ {\overset {\psi} \to}\  {A'' \to}\ 0$
be a right-exact sequence of finitely generated $G$-modules. Then we have
$\Sigma (M;A')\cap \Sigma (M;A'')\subseteq \Sigma (M;A)\subseteq \Sigma
(M;A'')$.  
\end{lemma}

\begin{proof} If $\epsilon :F\twoheadrightarrow A$ is a finitely generated
free presentation of the $G$-module $A$, we can compare it with the
free presentation $\psi \epsilon :F\twoheadrightarrow A''$ to prove
$\Sigma (M;A)\subseteq \Sigma (M;A'')$; note that finite generation of
$A'$ was not needed for this. Similarly, given two finitely generated
free presentations $\epsilon ' :F'\twoheadrightarrow A'$ and $\epsilon
'':F''\twoheadrightarrow A''$ we can lift $\epsilon ''$ to a homomorphism
${\widetilde \epsilon}'':F''\to A$, and consider the presentation $\phi
\epsilon '\oplus \widetilde {\epsilon}'':F'\oplus F'' \twoheadrightarrow
A$. The assertion now follows from the fact that each $a\in A$ can be
written as  $\phi \epsilon '(c')+
 \epsilon ''(c'')$ for some $(c',c'')\in F'\oplus F''$.
\end{proof}

\begin{cor}\label{dirsum}$\Sigma (M;A'\oplus A'')= \Sigma (M;A')\cap \Sigma (M;A'')$.
\hfill$\square$
\end{cor}

{\bf Induced modules:} Let $H$ be a subgroup of $G$, let $B$ be a finitely generated $G$-module,
and let $A={\Z}G\otimes _{H}B$ be the $G$-module induced by $B$. In this
situation the horospherical limit set $\Sigma ({_G}M;A)$ is determined
by $\Sigma ({_H}M;B)$ as follows:

\begin{thm}\label{subgroup} $\Sigma ({_G}M;A)$ is the largest $G$-invariant subset of 
$\Sigma ({_H}M;B)$. In other words we have 
\begin{equation}
\Sigma ({_G}M;{\Z}G\otimes _{H}B)=\bigcap_{g\in G}g\Sigma ({_H}M;B) 
\end{equation}
It is also the case that 
\begin{equation} 
\Sigmacirc ({_G}M;{\Z}G\otimes _{H}B)\subseteq \bigcap_{g\in G}g\Sigmacirc ({_H}M;B). 
\end{equation}
\end{thm}

\begin{proof}Let  $\epsilon :F\twoheadrightarrow B$ be a finitely
generated free presentation of the $H$-module $B$, and let ${\widetilde
\epsilon}:{\Z}G\otimes _{H}F\twoheadrightarrow A$ be the induced
presentation of $A$, where  ${\widetilde \epsilon}(g\otimes c)=g\otimes
\epsilon (c)$ for $g\in G$ and $c\in F$. We consider the canonical
control maps $h:F\to fM$ and ${\widetilde h}:{\Z}G\otimes _{H}F\to
fM$, noting that ${\widetilde h}(g\otimes c)=gh(c)$. The
presentation $\epsilon :F\twoheadrightarrow B$  appears as a canonical
direct summand in the presentation ${\widetilde \epsilon}$. Hence when
$b\in B$ is interpreted as $1\otimes b\in A$ and is represented as
${\widetilde \epsilon}(\Sigma t\otimes c_{t})=1\otimes b$, where $t$
runs through coset representatives of $G$ mod $H$ containing $1\in
G$, then $\epsilon (c_{1})=b$. As ${\widetilde h}(\Sigma t\otimes
c_{t})=\bigcup _{t}th(c_{t})$ contains $h(c_{1})$ as a subset, it follows
that when $e$ is a horospherical accumulation point of ${\widetilde h}({\widetilde
\epsilon}^{-1}(1\otimes b))$ then it is also a horospherical accumulatoin point
of $h({\epsilon ^{-1}}(b))$. This shows that $\Sigma ({_G}M;A)\subseteq \Sigma
({_H}M;B)$. The corresponding inclusion for $\Sigmacirc $ is obtained similarly by showing that a $G$-finitary endomorphism of ${\Z}G\otimes _{H}F$ induces an $H$-finitary endomorphism on the direct summand $F$; or, alternatively, by referring to Theorem \ref{T:6.1} below which expresses $\Sigmacirc $ in terms of $\Sigma $. The containments 
$\Sigma ({_G}M;A)\subseteq \bigcap_{g\in G}g\Sigma ({_H}M;B)$ and $\Sigmacirc ({_G}M;A)\subseteq \bigcap_{g\in G}g\Sigmacirc ({_H}M;B)$ now follow by $G$-equivariance.

To prove the converse (for $\Sigma $) let $e\in \bigcap _{g\in G}g\Sigma (_{H}M;B)$
and let $HB_{e}\subseteq M$ be a horoball at $e$. Given an element
$a\in A$ in the canonical expansion $a=\Sigma t\otimes b_{t}$,
where $t$ runs through coset representatives and $b_t\in B$, we
use $t^{-1}e\in \Sigma ({_H}M;B)$ to write $b_{t}=\epsilon (c_{t})$
with $h(c_{t})\subseteq HB_{t^{-1}e}=t^{-1}HB_e$. Then ${\widetilde
\epsilon}(\Sigma t\otimes c_{t})=a$ and ${\widetilde h}(\Sigma t\otimes
c_{t})=\bigcup th(c_{t})\subseteq HB_e$. This shows that $e\in \Sigma
({_G}M;A)$.  
\end{proof}

Whether or not the other inclusion $\supseteq $ holds for $\Sigmacirc $ is an intriguing question which may be difficult.

\begin{example} Let $A={\Z}\Omega $ be the permutation module of a $G$-set
$\Omega $ which has finitely many orbits. Corollary \ref{dirsum} and Theorem
\ref{subgroup} together give a decomposition of $\Sigma (_{G}M;{\Z}\Omega
)$ as $\bigcap _{i}\bigcap _{g}g\Sigma (_{H_{i}}M;{\Z})$, where $i$
ranges over the orbits and $H_i$ stabilizes a member of the $i$th orbit.
\end{example}

{\bf $\Sigma (M;A)$ in terms of finite generation:} For a horoball $HB\subseteq M$ we put
$$G_{HB}:=\{g\in G\mid h(g)\in HB\}$$
noting that $G_{HB}$ is not in general a monoid, and may be empty. The following is immediate from the definition of $\Sigma (M;A)$:

\begin{prop}\label{fg} The following are equivalent:
\begin{enumerate}[(i)]
\item $e\in \Sigma (M;A)$;
\item $A$ is finitely generated as an $HB$-operator group;
\item  $A = {\mathbb Z}G_{HB}{\mathcal A}$, for every finite 
${\mathbb Z}G$-generating subset ${\mathcal A}\subseteq A$ and all horoballs $HB$ at $e$.
\end{enumerate}
\hfill$\square$
\end{prop}

\begin{rems}\label{elementary}
\begin{enumerate}[(1)]
\item  $\Sigma (M;0)=\partial M$. 
\item If some (any) $G$-orbit in $M$ is bounded then $\Sigma (M;A)=\emptyset $.
\item When $A$ is free of positive rank then $\Sigma (M;A)=\emptyset .$
\item A non-zero module $A\neq 0$ can only be represented over $HB\subseteq M$ if $G_{HB}$ is non-empty, and in that case any $G$-module which is finitely generated as an abelian group is finitely generated over $G_{HB}$. This shows that $\Sigma (M;{\Z})$ contains $\Sigma (M;A)$ for every finitely generated $G$-module $A$, and coincides with $\Sigma (M;A)$ when $A$ is finitely generated over $\Z$.
\item The condition that $e\in \Sigma (M;{\Z})$ thus requires that
$G_{HB}$ be non-empty for every horoball $HB$ at $e$; i.e. that $e$ is
a horospherical limit point of the orbit $Gb$. When this holds for all
$e\in \del M$, Theorem 12.2 of \cite{BGe03} implies that the action of
$G$ on $M$ is cocompact. Thus we have:
\end{enumerate}
\end{rems}
\begin{prop}\label{half} If $A\neq 0$ and $\Sigma(M;A)=\partial M$ then the $G$-action on
$M$ is cocompact.
\hfill$\square$
\end{prop}

In Theorem \ref{uniform}, below, we will complete Proposition \ref{half} by giving necessary
and sufficient conditions for  $\Sigma (M;A)=\partial M$ in terms of cocompactness plus a bounded generation property.


\section{Characterization of $\Sigmacirc (M;A)$ in terms of $\Sigma
(M;A)$}\label{S:6}

In this section we characterize $\Sigmacirc(M;A)$ as a specific subset of
$\Sigma (M;A)$ (Theorem \ref{T:6.1}), and we give conditions
under which $\Sigmacirc(M;A)=\Sigma (M;A)$ (Theorem
\ref{T:6.4}).

\begin{thm}\label{T:6.1} For each finitely generated $G$-module $A$ we
have \begin{equation*} \Sigmacirc(M;A) = \{e\in \partial M\mid \text{\rm
cl}(Ge)\subeq \Sigma (M;A)\}.  \end{equation*} \end{thm}

This\footnote{When $A$ is the trivial $G$-module $\Z$, $\Sigma (M;{\Z})$
coincides with the homotopical invariant $\Sigma ^{0}$ of \cite{BGe03},
and Theorem \ref{T:6.1} specializes to Theorem E of that Memoir.}
shows that $\Sigmacirc(M;A)$ is determined by $\Sigma (M;A)$

The inclusion $\subeq$ of Theorem \ref{T:6.1} follows from Theorem
\ref{P:3.4} together with (\ref{sigmadef}).  We turn to the other
inclusion $\supeq$.

\begin{thm}\label{T:6.4} Let $F\thra A$ be a controlled based free
presentation, and let $E$ be a closed $G$-invariant subset of $\partial
M$. If $E\subeq \Sigma(M;A)$, then $E\subeq \Sigmacirc(M;A)$. Moreover,
there is a uniform constant $\nu >0$ and a finite $G$-volley $\Phi$
inducing $\text{\rm id}_A$ such that for each $e\in E$ there is a selection $\varphi
_{e} \in \Phi$ with ${\rm gsh}_{e}\varphi _{e}\geq \nu$.
\end{thm}

\begin{proof} We are given $E\subeq \Sigma (M;A)$.  For each $x\in X$
and $e\in E$ we choose $\bar c(e,x)\in F$ such that $\epsilon(\bar
c(e,x)) = \epsilon(x)$ and $v_{e}(\bar c(e,x)) -v_{e}(x) > 0$.  Since this inequality holds for $e$ and $x$, it also holds when $v_{e}$ is replaced by $v_{e'}$ provided $e'$ lies in a suitably small
neighborhood of $e$.  Since $E$ is compact there is a finite subset
$E_f\subeq E$ such that for each $e\in E$ there is some $e'\in E_f$ such
that $v_{e}(\bar c(e',x)) -v_{e}(x) > 0$.
For every $e\in E$ we choose such an $e'$ and define $c(e,x) := \bar
c(e',x)$. Thus $\ds{\inf_{e\in E}}\{v_{e}(c(e,x))-v_{e}(x)\} > 0$.
Define $\Psi(x) = \{c(e,x)\mid e\in E\}$ (which is a finite subset of $F$) and extend to get the associated canonical finite $G$-volley $\Psi :F\to fF$ inducing id$_A$.

For $e\in E$ and $y = gx$ define a ${\mathbb Z}$-endomorphism $\psi_{e} :
F \to F$ by $\psi_e(y) := gc(g^{-1}e,x)$; this makes sense because $E$
is $G$-invariant.  Then $\epsilon\psi_e = \epsilon$, and
\begin{equation*}
\begin{aligned}
v_{e}(\psi_e(y)) - v_{e}(y) &= v_{e}(gc(g^{-1}e,x)) -v_{e}(gx)\\
&= v_{g^{-1}e}(c(g^{-1}e,x)) - v_{g^{-1}e}(x) \text{ by Lemma }\ref{L:2.3}(iii).
\end{aligned}
\end{equation*}
Since each $\psi _e$ is a selection from $\Psi $, $E\subseteq \Sigmacirc
(M;A)$.  And since $\ds{\inf_{e\in E}}\{\text{gsh}_e(\psi_e)\} > 0$
the final sentence of the Theorem holds.

\end{proof}


\section{Varying the action of $G$}\label{S:6.4}

Let $E\subseteq \partial M$ and let ${\mathcal R}_{E}:=\text{Hom}(G,\text{Isom}(M,E))$ denote the set of all isometric
actions of $G$ on $M$ which leave $E$ invariant. We endow the set
$\text{Isom}(M,E)$ and the set ${\mathcal R}_{E}$ with the compact-open
topology. (Recall that $G$ is discrete.)  It is convenient to choose a
base point $b\in M$, so that we can write $v_e$ rather than $v_{\gamma}$.
In this section, when we discuss a particular action $\rho \in {\mathcal
R}_{E}$ we write $_{\rho}M$ rather than $M$, and $v_{e}^{\rho}$ rather than $v_e$ for the valuation at $e$ using the action $\rho$.The boundary $\del M$ carries the cone topology.

We first note that for given $g\in G$ the map $\eta _{g}:{\mathcal R}_{E}\to
M$ taking $\rho $ to $\rho (g)(b)$ is continuous in $\rho$. This is
because, for given $\epsilon >0$, it maps the open neighborhood of $\rho$ 

$$W_{\epsilon }:=\{\rho '\in {\mathcal R}_{E}| \rho '(g)(b)\in B_{\epsilon
}(\rho (g)(b)))\}$$ 

\noindent into the open ball $B_{\epsilon }(\rho (g)(b))$. The
map $\beta :M\times \del M\to {\R}$ taking $(p,e)$ to $\beta _{e}(p)$ is also
known to be continuous.  Hence the following composite map is continuous:

$$
\xymatrix{
&{{\mathcal R}_{E}\times \del M}\ar[r]^{\eta _{g}\times \text{id}}&{M\times \del M}\ar[r]^\beta & {\R}
}
$$

From this one gets:

\begin{lemma}\label{continuous} For given $c\in F$, $v^{\rho}_e(c)$ is (jointly) continuous in $(\rho ,e)$.
\hfill$\square$
\end{lemma}

\begin{thm}[$\text{Openness Theorem}$]\label{T:6.5} Let  $E$ be a
closed subset of $\partial M$, and let $\rho \in {\mathcal R}_{E}$
be such that $E\subseteq \Sigma (_{\rho}M;A)$. There is a neighborhood
$N$ of $\rho$ in ${\mathcal R}_{E}$ such that for all $\rho '\in N,$
$E\subseteq \Sigmacirc (_{\rho '}M;A)$. \end{thm}

\begin{proof} By Corollary \ref{invariant} we have $E\subseteq
\Sigmacirc (_{\rho}M;A)$. Fix $x\in X$. For each $e\in E$ there is
a finitary map $\varphi _{e,\rho }:F\to F$ lifting $\text{id}_A$,
and a number $\delta _{e}>0 $ such that $$v_{e}^{\rho }(\varphi _{e,
\rho }(x))-v_{e}^{\rho }(x) > \delta _{e}.$$ By Lemma \ref{continuous}
there is a product neighborhood $N(e, \rho )=N(e)\times N_{e}(\rho )$
such that when  $(e',\rho ')$ lies in $N(e, \rho )$ then $$v_{e'}^{\rho
'}(\varphi _{e,\rho }(x))-v_{e'}^{\rho '}(x) > \delta _{e}.$$ The
sets $\{N(e)\}$ form an open cover of $E$, so there is a finite
subcover $\{N(e_{i})\}$. We write $N:=\bigcap N(e_{i})$
and $\delta :=\text{min}\{\delta _{e_{i}}\}$. Then we have finitely many
finitary maps $\varphi _{i}:=\varphi _{e_{i}, \rho }$ such that for any
$e\in E$ and $\rho '\in N$ there exists $i$ satisfying 
$$v_{e}^{\rho '}(\varphi _{i}(x))-v_{e}^{\rho '}(x) > \delta.$$ 
Now fix $e\in E$. For
each $\rho '\in N$ pick such a $\varphi _{i}$ and call it $\varphi _{e,
\rho '}$. Then for all $\rho '\in N$ $$v_{e}^{\rho '}(\varphi _{e, \rho
'}(x))-v_{e}^{\rho '}(x) > \delta.$$ Define  $\psi _{e, \rho '}:F\to
F$ by $\psi _{e, \rho '}(gx):= g\varphi _{e, \rho '}(x)$. Each of
the (finitely many) finitary maps $\varphi _{i}$ comes with a
volley $\Phi _i$. Combine these to get a new volley defined by $\Phi (c)=\bigcup \Phi _{i}(c)$. Then the maps $\psi _{e, \rho '}$ are selections of $\Phi $ which lift
$\text{id}_A$, and their guaranteed shifts towards $e$ are $\geq \delta$.

\end{proof}

\begin{Rem} The above proof of Theorem \ref{T:6.5} proves more: there is a volley $\Phi :F\to fF$ and a number $\delta >0$ such that for every $(e, \rho ')\in E\times N$ there is a selection 
$\varphi _{e, \rho '}$ of $\Phi$ such that $\text{gsh}_{e}(\varphi _{e, \rho ' }\geq \delta $.
\end{Rem}

\begin{cor}\label{open2} Let $\rho $ be an isometric action on $M$ as above. There is
a neighborhood $N$ of $\rho$ in ${\mathcal R}_{\del M}$ such that if
$\Sigmacirc (_{\rho }M;A)=\partial M$ then $\Sigmacirc (_{\rho '}M;A)=\partial M$
for all $\rho '\in N$. 						
\hfill$\square$
\end{cor}


\section{The meaning of $\Sigma(M;A)=\partial M$.}\label{S:7}

In this section we assume that the finitely generated
${\Z}G$-module $A$ is non-zero, and we study the condition 
$\Sigma(M;A)=\partial M$.  By Theorem \ref{T:6.4}, the statements
$\Sigma(M;A)=\partial M$ and $\Sigmacirc(M;A)=\partial M$ are
equivalent. Our goal is Theorem \ref{uniform}.
It explains how this is equivalent to cocompactness together with the
property ``$A$ has bounded support over $M$". In the case of discrete
orbits the latter reduces to the algebraic property ``$A$ is finitely
generated over ${\Z}H$", where, depending on hypotheses, $H$ is either
the kernel of the action or the stabilizer of a point of $M$.

\subsection{Bounded support}\label{S:7.1}

The statement $e\in \Sigma (M;A)$ means that every $a\in A$ can be
``supported" over every horoball at $e$.  Here we will need the analogous
concept: support over a bounded subset of $M$.  We say the module $A$ has
{\it bounded support over} $M$ if there is a bounded subset $B\subseteq
M$ with the property that for each $a\in A$ there exists $c\in F$ with
$\epsilon(c) = a$ and $h(c)\subseteq B$.  It is easy to see that this
property is independent of the choice of $F$ and of the control map
$h$. When this property holds over every bounded set of a particular
diameter we say that $A$ has {\it uniform bounded support} over $M$.

\begin{thm}\label{uniform} The following are equivalent:
\begin{enumerate}[(i)]
\item $\Sigma(M;A)=\partial M$;
\item The action $\rho $ is cocompact and $A$ has bounded support over $M$;
\item $A$ has uniform bounded support over $M$.
\end{enumerate}
\end{thm}

\begin{proof} The equivalence of (ii) and (iii) is clear. Since there is
an arbitrarily large ball inside any horoball, (ii) implies (i). That (i)
implies cocompactness is Proposition \ref{half}.  The remaining item,
the fact that (i) implies bounded support, requires some work and will
be proved in Section \ref{completion}.  
\end{proof}

The example of $SL_{2}({\mathbb Z})$ acting on the hyperbolic plane,
where we take $A$ to be the trivial $G$-module ${\mathbb Z}$, shows that
``having bounded support" does not imply ``cocompact".

\subsection{Shifting towards a point of $M$}\label{S:7.3}

Just as we needed the idea of pushing towards $e\in \partial M$, now we
need the analogous idea of pushing towards a point $b\in M$. The role
of a valuation on $F$ is played by the function $D_b :F\to {\mathbb
R}_{\geq 0}$ defined by  $D_b(c) :=\max \{d(p,b)\mid p\in h(c)\}$ when
$c\neq 0$ and $D_b(c) = 0$ when $c=0$. 

With notation as before, let $\varphi : F\to F$ be a
$\Z$-endomorphism. The {\it shift function of} $\varphi $ {\it towards}
$b\in M$ measures the loss of distance from $b$ (over $M$) of elements of $F$;
it is denoted by sh$_{\varphi,b} : F\to {\mathbb R}$, and is defined
by 

\begin{equation}\label{E:7.1} \text{sh}_{\varphi,b}(c) := D_b(c)
- D_b(\varphi(c)) \in {\mathbb R}\cup \{\infty\}
\end{equation}

The notion of guaranteed shift towards $b\in M$ is more subtle than
the corresponding notion for  endpoints $e\in \partial M$ because if
elements are already too close to $b$ it may not be possible to push
them any closer.  Therefore we have to restrict attention to elements $c$
with $h(c)$ outside some ball centered at $b$.  When $t \in {\mathbb R}$
and $R\geq 0$, the pair $(t,R)$ {\it defines a guaranteed shift of}  
$\varphi $ {\it towards} $b$ if sh$_{\varphi,b}(c)\geq t$ whenever $c\in F$ and $D_b(c)
> R$.  The {\it almost guaranteed shift of} $\varphi$ {\it towards} $b$ is
\begin{equation*}
\text{gsh}_b(\varphi ) := \sup\{t\mid \text{for some } R, (t,R) \text{ defines a guaranteed shift
for }\varphi \}.
\end{equation*}

\begin{lemma}\label{L:7.2}
\begin{enumerate}[{\rm (i)}]
\item $-||\varphi || \leq \text{\rm gsh}_b(\varphi ) \leq ||\varphi ||$.
\item If $\psi : F\to F$ is another  $\Z$-endomorphism then
\begin{equation*}
\text{\rm gsh}_b(\varphi \circ \psi ) \geq \text{\rm gsh}_b(\varphi ) + \text{\rm gsh}_b(\psi ).
\end{equation*}
\end{enumerate}
\end{lemma}

\begin{proof}  (i) is clear.  For (ii) let $(t,R(t))$ and $(t',R(t'))$
define guaranteed shifts for $\psi $ and $\varphi $ respectively.
For all $c\in F$ with $D_{b}(c)>R(t)$ we have sh$_{\psi ,b}(c)\geq t$,
and by (i) we have $D_b(\psi (c))>R(t )-||\psi ||$. Thus, arguing as in
the proof of Lemma \ref{L:3.3}, we find that 
$R:=\text { max}\{R(t), R(t')+||\psi ||\}$ will be such that the
pair $(t+t',R)$ defines a guaranteed shift for $\varphi \circ \psi $.
\end{proof}

We note that when $\varphi$ is $G$-finitary, $||\varphi || <\infty$ and
gsh$_{b}(\varphi )$ is attained.  If gsh$_b(\varphi ) > 0$ we say that
$\varphi $ {\it pushes} $F$ {\it towards} $b\in M$.

\begin{cor}\label{C:7.3} If $\varphi$ in Lemma \ref{L:7.2}(ii) pushes $F$
towards $b$, then $\varphi^k\circ \psi$ pushes $F$ towards $b$ whenever $k
> \dfrac{-\text{\rm gsh}_b(\psi )}{\text{\rm gsh}_b(\varphi )}$.  In fact,
for any $\eta >0$, {\rm gsh}$_b(\varphi^k\circ\psi ) > \eta$ when $k >
\dfrac{\eta-\text{\rm gsh}_b(\psi )}{\text{\rm gsh}_b(\varphi )}$.
\end{cor}

\subsection{$CAT(0)$ issues}

\begin{lemma}\label{neighborhood2} Let $p\in M$, let $\gamma$ be a geodesic ray starting at $p$, let $r>0$ and $\epsilon >0$ be given, let $R>r(1+\frac{2r}{\epsilon})$, let $q\in B_{r}(p)$ and let $b=\gamma (R)$. Then 
\begin{equation}\label{star2}
d(p,b)-d(q,b)\geq \beta _{\gamma}(q)-\beta _{\gamma}(p)-\epsilon
\end{equation}
\end{lemma}
\begin{proof} This is an immediate consequence of (\ref{star}) in the proof of Lemma \ref{neighborhood}.
\end{proof}

The $CAT(0)$ metric space $M$ is {\it almost geodesically complete} if
there is a number $\mu\geq 0$ such that for any $b$ and $p\in M$ there
is a geodesic ray $\gamma$ starting at $p$ and passing within $\mu $ of $b$.
(An example lacking this property is the half line $[0,\infty)$.)  
It is a theorem in \cite{GO07} that whenever the isometry group of $M$
acts cocompactly then $M$ is almost geodesically complete.

\begin{prop}\label{P:7.6}
Let $M$ be almost geodesically complete.  The following are equivalent for a $G$-volley $\Phi:F\to fF$:
\begin{enumerate}[{\rm (i)}]
\item $\forall e\in \partial M$ $\Phi$ admits a selection pushing $F$ towards $e$ which induces 
$\text{\rm id}_A$.
\item $\forall b\in M$ $\Phi$ admits a selection pushing $F$ towards $b$ which induces 
$\text{\rm id}_A$.
\end{enumerate}
\end{prop}

\begin{proof} We first prove (i) $\Rightarrow$ (ii). We are to show that there is a selection pushing $F$ towards a given $b\in M$. We use the canonical control map, so for $y\in Y(=GX)$ the set $h(y)$ is a singleton in $M$. Define $f : \partial M\to {\mathbb R}$ by 
\begin{equation*}
f(e) = \max\{v_\gamma(c)-v_\gamma(x)\mid x\in X, c\in \Phi(x)\}
\end{equation*}
where $\gamma$ is any geodesic ray defining $e$.  
By (i), $f(e) > 0$.  Since $f$ is continuous and $\partial M$ is compact there exists $\delta >
0$ such that $f(\partial
M) > \delta$.  
We can write $$f(e) = \max\{v_\gamma(\Phi(x)) - v_\gamma(x)\mid x\in X\}$$
\noindent For $y=gx$ we have 

\begin{equation*} \max\{v_\gamma(\Phi(gx)) -
v_\gamma(gx)\} = \max\{v_{g^{-1}\gamma}(\Phi(x)) - v_{g^{-1}\gamma}(x)\}
> \delta   
\end{equation*} 

Let $\mu$ come from the definition of ``almost geodesically complete". For each $y = gx$ let $\gamma_y$ be a geodesic ray starting at $h(y)$ and passing within $\mu $ of $b$. For simplicity we first assume that $b$ lies on that ray; a slight adjustment, given below, deals with the general case. Write $e(y):=\gamma (\infty)$. Choose $\psi(y) \in \Phi(y)$ so that
$v_{\gamma_y}(\psi(y)) - v_{\gamma_y}(y) \geq \delta $.  This defines
a selection $\psi : F\to F$ from $\Phi$ with sh$_{\psi,e(y)}(y)\geq
\delta $.  We claim $\psi$ pushes $F$ towards $b$.  To see this,
apply Lemma \ref{neighborhood2} with $r=||\Phi ||$, $\epsilon  =
\frac {\delta }{2}$, $p=h(y)$ and $q$ a point in the set $h(\psi (y))$. Then when $d(p,b)\geq R$ we have $d(p,b)-d(q,b)\geq \frac{\delta}{2}$. Since this holds for all $q\in h(\psi (y))$ the claim is proved.

In general, this push is not towards $b$ but towards a point $\bar b$ in the $\mu$-ball about $b$. We then have 
$$d(p,b)-d(q,b)-2\mu \leq d(p,{\bar b})-d(q,{\bar b})$$
If $p$ (and hence $q$) are far enough from $b$ and if $k\delta \geq 2\mu$ then by Corollary \ref{C:7.3} $\psi ^{k}$ has almost guaranteed shift $\frac{\delta}{2}$ towards $b$.

(ii) $\Rightarrow$ (i): This is immediate because the property ``almost geodesically complete" is uniform, so the ball of radius $\mu $ can be located inside any horoball. 
\end{proof}

\subsection{Completion of proof of Theorem \ref{uniform}}\label{completion}

We assume (i) and we know that this implies cocompactness.  Hence, by the
theorem of \cite{GO07} mentioned above, it follows that $M$ is almost
geodesically complete. So, for any $b\in M$, Theorem \ref{T:6.4} and Proposition \ref{P:7.6}
give us a finite $G$-volley $\Phi $ having a selection $\varphi $, inducing $\text{id}_A$, with
${\text{\rm gsh}_b(\varphi)}>0$.  Let $(\alpha , R)$ define a guaranteed
shift for $\varphi$, where $\alpha >0$.  For any $a\in A$ there exists 
$c\in F$ mapped by $\epsilon $ to $a$ such that 
$$D_{b}(\varphi (c))\leq {\text{\rm max }}(R+||\varphi ||,
D_{b}(c)-\alpha).$$
Corollary \ref{C:7.3} then implies that by iterating $\varphi $ we can move $c$ over $M$ to a
new $c'$ such that $\epsilon (c')=a$ and $h(c')$ lies over the ball centered at $b$ with radius
$R+||\varphi ||$, a number independent of $a$. This completes the proof of Theorem \ref{uniform}.

\subsection{Finite generation over a smaller ring}

For $b\in M$ we write $G_b$ for the subgroup of $G$ fixing $b$. Note
that when $G$-orbits are discrete the group $G_{b'}$ is commensurable
with $G_b$, for any $b'\in M$.

\begin{cor}\label{smaller} Let $b\in M$. Assume that the $G$-orbits are
discrete subsets of $M$. Then $\Sigma(M;A)=\partial M$ if and only if
the $G$-action on $M$ is cocompact and $A$ is finitely generated as a $G_b$-module.
\end{cor}

\begin{proof} By Theorem \ref{uniform}, $\Sigma(M;A)=\partial M$ if
and only if the $G$-action is cocompact and $A$ has bounded support
over $M$. Filter $ F$ by $h^{-1}(fB_{m}(b))$, where $m\geq 1$ 
and the notation means the largest ${\mathbb Z}$-subcomplex mapped by $h$
into $fB_{m}(b)$.  Because the orbits are discrete these abelian subgroups
provide a filtration of $F$ by finitely generated $G_b$-modules. By an
obvious adaptation of Theorem 2.2 of \cite{Bn87} the existence of this
filtration is equivalent $A$ being  finitely generated as a $G_b$-module.
\end{proof}

Let $\rho :G\to \text{Isom}(M)$ denote the $G$-action on $M$. A variant of
Corollary \ref{smaller} is:

\begin{cor}\label{kernel} Assume that the group $\rho (G)$ acts properly discontinuously
and cocompactly (aka ``geometrically") on $M$. Then $\Sigma(M;A)=\partial
M$ if and only if $A$ is finitely generated as a ${\mathbb Z}[{\rm ker}\rho ]$-{\rm module}.  
\end{cor} 
\begin{proof} The hypothesis implies that $N$ and $G_b$ are commensurable.
\end{proof}

\begin{Rem} By Corollary \ref{open2} the condition that point stabilizers
or kernels are finitely generated (under the hypotheses of Corollary
\ref{smaller} or Corollary \ref{kernel}) is an open condition with
respect to the action $\rho $.  
\end{Rem} 


\section{Hyperbolic considerations}

\subsection{The case of Gromov-hyperbolic $CAT(0)$ spaces}\label{hyperbolic2}

As in Section \ref{hyperbolic}, we assume $G$ acts on the Gromov-hyperbolic proper $CAT(0)$ space $M$ by isometries, that some (hence every) $G$-orbit is unbounded, and that the given finitely generated $G$-module $A$ is non-zero. 

Some properties of such a space $M$ are 
 
{\it H1} Any two points of $\partial M$ are the endpoints of a line in $M$; 
 
{\it H2} For every point $e\in \partial M$ there is a basis of open cone-neighborhoods $\{N_i\}$ of $e$ in $M\cup \partial M$ such that $N_{i}\cap M$ contains a horoball at $e$.
 
{\it H3} Each point of $\partial M$ has a basic system of neighborhoods $\{N_{i}\}$ in $M\cup \partial M$ such that the geodesic joining any two points of $M-N_{i}$ lies in $M-N_{i+1}$.

Our goal is to understand $\Sigmacirc (M;A)$ in this situation (Theorem \ref{alternatives2}).

The notation $\widehat M$, as well as the terms ``interval in $\widehat M$", and ``closed convex hull"  were defined in Section \ref{hyperbolic}. As before, we write $\Lambda (G)$
for the cone topology  limit set of $G$ in $\del M$; i.e. $\Lambda (G)=\Lambda (M;{\Z})$.

A version of {\it H3} for $\widehat M$ reads:

{\it $H3^{\prime }$} Each point of $\partial M$ has a basic system
of neighborhoods $\{N_{i}\}$ in ${\widehat M}$ such that the interval
joining any two points of ${\widehat M}-N_{i}$ lies in ${\widehat
M}-N_{i+1}$. (Simply truncate an interval which has one or both end
points in $\partial M$ and apply {\it H3}.)

In our situation we have:

(1)\;\;\;\;\;\;\;\;\;\;\;\;If $S$ is closed in $\widehat M$ so is $S[2]$.

To see (1), consider a limit point $p$ of $S[2]$. If
$p\in M$ then there is a sequence of intervals containing points
$p_i$ converging to $p$. Since $S$ is closed, 3.10 of \cite{CS98}
implies\footnote{This is needed because there can be more than one
line joining two points of $\partial M$; a Zorn's Lemma argument picks out the desired sequence.} $p\in S[2]$. Next, let
$p\in \partial M$ and suppose $p$ does not lie in $S[2]$. Then $p\notin S\cap
\partial M$. Pick a neighborhood $U$ of $p$ disjoint from $S$. By {\it $H3^{\prime }$} there is a smaller neighborhood $V$ of $p$ such that no interval with end points in $S$ meets $V$. This contradicts the fact that $p$ is a limit point of $S[2]$.

Our next observation is: 

$$(2)\;\;\;\; \;\ \text {If }K\subseteq \Lambda (G) \text { is closed and } G\text {-invariant, then } K \text{ is empty, is a singleton, or }K= \Lambda (G).$$

\noindent To see (2), assume $|K|\geq 2$, and choose $b$ lying in
the (non-empty) set $ K[2]$. Then $Gb \subseteq K[2]$.
Thus, by (1), 
$$\Lambda (G)=\partial M\cap \text{cl}_{\widehat M}Gb\subseteq \partial M\cap K[2]=K\subseteq \Lambda (G).$$

\begin{lemma}\label{character}Let $Ge=\{e\}$. For $g\in G$, the number
$\chi_{e}(g):=\beta _{e}(gx)-\beta _{e}(x)$ is independent of $x\in M$;
$\chi_{e}:G\to {\R}$ is a homomorphism to the additive group of reals. The
element $g$ acts on $M$ as a hyperbolic isometry if and only if $\chi
_{e}(g)\neq 0$. The point $e$ is an end-point of an axis\footnote{All axes of a hyperbolic isometry are parallel, p.231 of \cite{BrHa99}, so they have the same two endpoints.} of any such
hyperbolic isometry.  
\end{lemma}

\begin{proof} It is straightforward to show that $\chi$ is a well-defined homomorphism and that
$\chi(g)\neq 0$ implies $g$ is hyperbolic. It follows
that when $b\in M$ and $\chi (g)\neq 0$ then a power of $g$ moves $b$
into any given horoball at $e$. This would be impossible if neither endpoint
of an axis of $g$ were $e$.  
\end{proof}

(3)\;\;\;\;\;\;If $e\in \partial M$ is fixed by $G$ then $\Lambda (G)=\{e\}$ if and only if no member of $G$ acts on $M$ as a hyperbolic isometry.   

To see (3), note first that when $g$ acts as a hyperbolic isometry then the two end points of its axis are limit points; hence $\Lambda (G)\neq \{e\}$. Conversely, if no member of $G$ acts as a hyperbolic isometry, then  $\chi _{e}$ (in Lemma \ref{character}) is identically $0$,
which implies that $G$ leaves every horoball at $e$ invariant. Any
$b\in M$ lies on a horosphere at $e$. The entire orbit $Gb$ also lies in that horosphere,
so, by II9.35(4) of \cite{BrHa99}, no other point of $\partial M$ can
be in $\Lambda (G)$. By the standing assumption, orbits are unbounded so $\Lambda
(G)\neq \emptyset$, hence $\Lambda (G)=\{e\}$.

A further observation follows from {\it H2}: 
$$(4)\;\; \text{ for any non-zero finitely generated }G\text{-module }A\text{, } \Sigma (M;A)\subseteq \Lambda (M;A).$$

We can now state the possibilities when $G$ fixes a point of $\partial M$. 

\begin{prop}\label{fix} If $G$ fixes $e\in \partial M$ either 
\begin{enumerate}[(a)]
\item $\Lambda (G)=\{e\}$, in which case $Gb$ lies on a horosphere at $e$, and $\Sigma (M;{\Z})$ is empty, or
\item $\Lambda (G)$ contains $e$ and at least one other point, in which
case $\Sigmacirc (M;{\Z})= \Sigma (M;{\Z})=\Lambda (G)$.
\end{enumerate}
\end{prop}

\begin{Rem}\label{singleton} It follows from Proposition \ref{fix} that $\Sigma (M;{\Z})$ can never be a singleton.	
\end{Rem}

\begin{proof}(of the Proposition) Part {\it (a)} follows from the proof of (3).  For {\it
(b)}, we know by (3) that $G$ contains some $h$ which acts as a hyperbolic
isometry. By Lemma \ref{character}, an axis $L_h$ of $h$ connects $e$ with
another boundary point $e'$. We choose $b\in L_h$. Each point $ge'$ in the
$G$-orbit of $e'$ is the endpoint of an axis of an element $h'=ghg^{-1}$,
and this axis carries a sequence of orbit points $gh^{i}b, i\in {\Z}$,
with constant distance $d(b,hb)$ between neighboring members of the
sequence. Now, if $e''$ is a point in the closure of $Ge'$, and $e''\neq
e$, then the line from $e$ to $e''$ is a limit of such axes. This shows
that every horoball at $e''$ contains points of the orbit $Gb$. The
same holds for $e$, since it is an endpoint of $L_h$. Hence cl$(Ge')$
is contained in $\Sigma (M;{\Z})$ and therefore in $\Sigmacirc (M;{\Z})$
by Theorem \ref{closure}.

If $Ge'\neq \{e'\}$ then by (2) we know that cl$(Ge') =\Lambda (G)$, and so
the assertion of {\it (b)} follows from (4) and Theorem \ref{closure}. If $Ge'=e'$ then 
the whole orbit of $b$ lies in $L_h$, and 
$e$ and $e'$ are the only points of $\Lambda (G)$.
\end{proof}

Proposition \ref{fix} does not hold when the trivial module $\Z$ is replaced by a general module $A$. We will give an example in Section \ref{H2}.

\begin{thm}\label{alternatives2}If the proper $CAT(0)$ space $M$ is Gromov-hyperbolic and if $A$ is a (non-zero) finitely generated $G$-module then $\Sigmacirc (M;A)$ is either empty, or is a singleton set, or coincides with the limit set $\Lambda (M;A)$.  Moreover, the following are equivalent:
\begin{enumerate}[(i)]
\item $\Sigmacirc (M;A)$ contains at least $2$ points;
\item $\Sigmacirc (M;A)=\Sigma (M;A)=\Lambda (M;A)=\Lambda (G)\neq \emptyset $; 
\item The weak convex hull of $\Lambda (G)$ is cocompact, and $A$ has bounded support over $\Lambda [2]$.
\end{enumerate}
\end{thm}

\begin{proof} $(i)\Rightarrow (ii)$:  Assume $|\Sigmacirc (M;A)|\geq 2$. If there is
$e\in \Sigmacirc (M;A)$ such that  $|Ge|\geq 2$, then (2) implies
cl$(Ge)=\Lambda (G)$.  Thus, by Theorem \ref{T:6.1}, (4), and Remark \ref{limitset} the conclusion holds. On the other hand, if $G$ fixes every member of
$\Sigmacirc (M;A)$, and $K$ is a two-element subset, then we can choose $b\in K[2]$, implying $\Lambda (G)=K$. 
 
$(ii)\Rightarrow (i)$: $\Sigma (M;A)\subseteq
\Sigma (M;{\Z})\subseteq \Lambda (G)$, so if $|\Sigmacirc (M;A)| =1$
then the hypothesis would imply $|\Sigmacirc (M;{\Z})|=1$ contradicting
Proposition \ref{fix}; and if $\Sigmacirc (M;A)$ were empty, then $\Lambda
(G)$ would be empty, contrary to {\it (ii)}.  
  
$(ii)\Rightarrow (iii)$: The proof uses two lemmas stated
below. That $\Lambda (G)[2]$ is cocompact follows from Lemma \ref{cocompact}. That $A$ has bounded support over $\Lambda (G)[2]$
follows from the $(i)\Rightarrow (ii)$ part of the proof\footnote{Since only two boundary points are involved, this is much simpler than Proposition \ref{P:7.6}.} of Proposition \ref{P:7.6}, where the hypothesis of almost geodesic completeness is replaced by the conclusion of Lemma
\ref{extendible}.
  
$(iii)\Rightarrow (ii)$: Since $A$ is non-zero and has bounded support over $\Lambda (G)[2]$, $\Lambda (G)$ must have more than one point. Hence for each $e\in \Lambda (G)$ there is a ray (in fact a line) in $\Lambda (G)[2]$ ending at $e$. Pick $b\in \Lambda (G)[2]$. There
is a ball centered at $b$ whose intersection with $\Lambda (G)[2]$ is a compact fundamental domain. Given any horoball at $e$, its intersection with $\Lambda (G)[2]$ can be covered with translates of this ball, hence $A$ can be represented over any horoball at $e$. Thus $\Sigma (M;A)=\Lambda (G)$. Since $\Lambda (G)$ is compact,  $\Sigmacirc (M;A)=\Sigma (M;A)$.
\end{proof}

Here are the two lemmas used in the proof of $(ii)\Rightarrow (iii)$:

\begin{lemma}\label{cocompact} Let $S$ be a closed $G$-invariant subset of $\widehat M$. If 
$S\cap \partial M\subseteq \Sigma(M;{\Z})$ then $S\cap M$ is cocompact.
\end{lemma} 

\begin{proof} Suppose $S\cap M$ is not cocompact. Then $S\cap M\neq
\emptyset$ so we can pick $b\in S\cap M$. For every positive integer
$n$ there exists $y_{n}\in S$ such that $B_{n}(y_{n})\cap Gb=\emptyset $. Let
$r_{n}>n$ be the critical radius such that $B_{r_{n}}(y_{n})\cap
Gb\neq \emptyset $, while $\text{int}B_{r_{n}}(y_{n})\cap Gb=
\emptyset$. Translating each $y_n$ by an element of $G$, we get
a sequence $(x_{n})$ in $S$ such that $b\in B_{r_{n}}(x_{n})$ and
$\text{int}B_{r_{n}}(x_{n})\cap Gb= \emptyset$. 
The sequence of $(x_n)$ is unbounded, hence, passing to a subsequence, we may assume it converges to some $e\in S\cap \partial M$. By hypothesis $e$ must lie in $\Sigma (M;{\Z})$. To see that this is false, let $\epsilon >0$. There is an integer $N$ such that for all $n\geq N$ the interval $[b, x_{n}]$ contains a point $z_{n}$ in the $\epsilon$-neighborhood of the ray $[b, e]$ such that $s_n:=d(b, z_{n})$ goes to infinity with $n$. We have $B_{s_{n}}(z_{n})\subseteq B_{r_{n}}(x_{n})$ so $b$ lies in the boundary of  $B_{s_{n}}(z_{n})$ while the orbit $Gb$ misses the interior of that ball. Pick $p_n$ on the ray $[b, e]$ with $d(z_{n},p_{n})\leq \epsilon$. Then the ball about $p_n$ of radius $d(p_{n},b)-2\epsilon $ misses the orbit $Gb$. Let $c$ be the point on the ray $[b,e]$ distant $2\epsilon $ from $b$. The horoball $HB_{e,c}$ misses $Gb$.Thus $e\notin \Sigma (M;{\Z})$.
\end{proof}

\begin{lemma}\label{extendible} Let $e\neq e'\in \partial M$, and let $b\in M$. There is a number $\mu  >0$ such that for every point $p\in M$ one of the geodesic rays $[p,e)$ and $[p,e')$ meets $B_{\mu }(b)$.
\end{lemma}

\begin{proof} By {\it H1} there is a line $\ell$ joining $e'$ to $e$.
Let $\nu $ be large enough that $B_{\nu}(b)$ meets $\ell$. The required $\mu $ is $\nu +2\delta$. For a contradiction, suppose there is $p\in M$ such that
neither $[p,e)$ nor $[p,e')$ meets $B_{\mu }(b)$. There are two unbounded
components in $\ell - B_{\nu }(b)$; we choose points $q$ and $q'$ far out on $\ell$
towards $e$ and $e'$, where the meaning of ``far out" is determined
as follows (for $q$; $q'$ is done similarly):
\begin{enumerate}
\item Using Lemma III H 3.3 of \cite{BrHa99}, pick a point $r$ on $[p,e)$ so far out that it is within $5\delta $ of $\ell$, and pick $q\in \ell$ to be within $5\delta$ of $r$.
\item Ensure that $B_{7\delta}(q)\cap B_{\nu }(b)=\emptyset.$
\end{enumerate}
Consider the geodesic triangle $(pqr)$. We have $[q,r]\subseteq
B_{5\delta}(q)$, so $N_{\delta}([q,r])\subseteq B_{6\delta}(q)$. By
hyperbolicity,
$$[p,q]\subseteq N_{\delta}([p,r])\cup B_{6\delta}(q).$$
Hence
$$N_{\delta}([p,q])\subseteq N_{2\delta}([p,r])\cup B_{7\delta}(q).$$
A similar statement holds when $q$ is replaced by $q'$. Now consider the geodesic triangle
$(pqq')$.  By hyperbolicity, we have
$$[q,q']\subseteq N_{\delta}([p,q])\cup N_{\delta}([p,q']).$$
But $N_{2\delta}([p,r])\cup B_{7\delta}(q)\cup N_{2\delta}([p,r'])\cup
B_{7\delta}(q')$ is disjoint from $B_{\nu }(b)$. So $[q,q']$ is disjoint
from $B_{\nu }(b)$. But this is false, since $q$ and $q'$ are separated in $\ell$ by 
$B_{\nu }(b)$.
\end{proof} 		

When $G$ acts on $M$ with discrete orbits, then the phrase ``$A$ has bounded support over $M$" becomes ``$A$ is finitely generated over the point stabilizer $G_b$". Thus we have:

\begin{cor}\label{finitely generated2} With $M$ and $A$ as in Theorem
\ref{alternatives2}, assume $G$-orbits in $M$ are discrete. Then $\Sigma (M;A)=\Lambda (G)$ if and only if the weak convex hull of $\Lambda $ is cocompact and $A$
is finitely generated over the stabilizer $G_b$ of $b$.
\end{cor}

\begin{cor}\label{finite type1} With $M$ as in Theorem \ref{alternatives2}, assume $G$-orbits in $M$ are discrete and that the stabilizer $G_b$ of some point $b$ is finite (i.e. the $G$-action is properly discontinuous). If $\Sigmacirc (M;{\Z})$ is non-empty then $G$ is of type $F_{\infty}$.
\end{cor}

\begin{proof} Since $\Sigmacirc (M;{\Z})$ cannot be a singleton, we have 
$$\Sigmacirc (M;{\Z})=\Sigma (M;{\Z})=\Lambda (G)$$  
By Lemma \ref{cocompact} the sets $\Lambda (G)[n]\cap M$ form an increasing sequence of $G$-subsets of $M$ each of which is cocompact. The increasing convexity properties of these sets shows that, given $k$,  we may assume, passing to a subsequence, that each inclusion of $\Lambda (G)[n]\cap M$ into $\Lambda (G)[n+1]\cap M$
is $(k+1)$-connected. If these $G$-spaces $\Lambda (G)[n]$ were $G$-complexes, the hypothesis on the stabilizers would allow us to apply Brown's Criterion \cite{Bn87} to conclude that $G$ is of type $F_k$, and hence $F_{\infty}$ because $k$ is arbitrary. 

Since they are not complexes, this argument needs some refinement. We
abbreviate $\Lambda (G)[n]\cap M$ to $L_{n}$. We first discuss the
case where $G$ acts freely as well as properly. Let $\mathcal U$ be a
finite open cover of the compact space $G\backslash L_n$ by sets small enough that their
closures are evenly covered by closed balls in $L_n$.
We denote the nerve of $\mathcal U$ by $N_n$, a finite complex of dimension, say, $d$. There is a  ``canonical map"\footnote{The construction of $g$ and some other
details omitted here are described in detail in the proof of Proposition
$A$ of \cite{On05}.} $g:G\backslash L_{n}\to N_n$. We let $K_{n}$ denote
the first barycentric subdivision of $N_n$. A map $f:K_{n}\to G\backslash
L_{n+d}$ can be defined as follows: each vertex $v\in K_{n}$ is mapped
to a point $x_v$ in the relevant intersection of members of $\mathcal
U$. This is then extended skeleton by skeleton to all of $K_{n}$ using
(short) geodesic coning at each stage. (For example, a simplex of the
$1$-skeleton is mapped to a short geodesic joining the images of its two
vertices.) One then constructs a homotopy $H$ in $G\backslash L_{n+d+1}$
between $f\circ g$ and the inclusion $G\backslash L_{n}\to G\backslash
L_{n+d+1}$ where every track of the homotopy is a short geodesic.

The cover $\mathcal U$ defines a $G$-cover $\widetilde {\mathcal U}$
of $L_n$ whose nerve $\widetilde {N_n}$ and its first barycentric
subdivision $\widetilde {K_n}$ are $G$-complexes.  There are lifts
${\tilde g}:L_{n}\to {\widetilde N_n}$,  ${\tilde f}:{\widetilde K_{n}}\to
L_{n+d}$ and a homotopy $\tilde H$ between ${\tilde f}\circ {\tilde
g}$ and the inclusion $L_{n}\to L_{n+d+1}$, where every track of the
homotopy is a short geodesic. Thus, identifying $\widetilde {N_n}$ with
$\widetilde {K_n}$ as topological spaces, we get a commutative diagram
of $G$-spaces and $G$-maps 

$$ \xymatrix{ &K_{n}\ar[r]^{{\tilde g}\circ
{\tilde f}} \ar[dr]^{{\tilde f}} &K_{n+d+1}\\ &{L_  {n}} \ar[u]^{{\tilde
g}} \ar[r]^{\iota}
     &{L_{n+d+1}}\ar[u]_{{\tilde g}}
} $$ 

Brown's Criterion can now be applied to $K_n$'s.

When the action is not free but merely proper, a well-known variant of this argument applies. The details are omitted as they can easily be extracted from the  proof of Proposition $A$ of \cite{On05}.
 
 \end{proof}

\subsection{The case when $M$ is hyperbolic $n$-space ${\H}^n$ }\label{Hn}

Here, $M={\H}^n$ and $\Gamma $ is an infinite discrete subgroup of Isom(${\H}^n$). We discuss how $\Sigma ({\H}^{n};\Z)$ and $\Sigmacirc ({\H}^{n};\Z)$ are related to standard properties of discrete hyperbolic groups. 

A point $e\in \partial {\H}^n$ is a {\it parabolic fixed point} if it is fixed by a
parabolic element of $\Gamma $; all parabolic fixed points lie in 
the limit set $\Lambda (\Gamma )$. A limit point $e$ is {\it conical} if there is a sequence of orbit points in ${\H}^n$ converging to $e$ which lie at a bounded distance from a geodesic ray ending at $e$; all conical limit points are horospherical limit points\footnote{See Theorem 15 of \cite{NW92} for an example of a Kleinian group having horospherical limit points which are not conical. Another example is given on page 95 of \cite{Ap00}.}. One knows that parabolic fixed points cannot be conical limit points. In the Fuchsian case ($n=2$) it is a fact that no parabolic fixed point is a horospherical limit point\footnote{Suppose $e=\gamma (\infty )$ is horospherical, and also is fixed by the parabolic element $p$. Then a sequence of orbit points converges to $e$ through horoballs, and, applying powers of $p$ to these, one can get them all to be within a bounded distance from the image of $\gamma$ (because $<p>$ acts cocompactly on horocircles at $e$). But it is well-known (see, for example, Lemma 3.1.2 of \cite{Bo93}) that conical limit points cannot be parabolic fixed points.}. We do not know if this holds in general, but it does hold in an important special case:

\begin{lemma}\label{parabolic} If a parabolic fixed point $e\in \Lambda
(\Gamma)$ is bounded (in the sense of Section 3 of \cite{Bo93}), then
$e\notin \Sigma ({\H}^{n};\Z)$.  
\end{lemma}

\begin{proof} By Proposition 4.4 of \cite{Bo93}, one can associate to
each bounded parabolic fixed point $e$ a standard parabolic region
$C(e)\subseteq  {\H}^{n}\cup \partial {\H}^{n}-\Lambda (\Gamma )$ such
that the collection of these regions is strictly invariant under $\Gamma
$; see \cite{Bo93}. These regions contain a strictly invariant collection of horoballs, so they cannot contain points of the orbit $\Gamma b$.  
\end{proof}

\begin{thm}\label{conical1} If $\Gamma $ is geometrically finite (in the sense of \cite{Bo93}) then $\Sigma ({\H}^{n};\Z)$ is the set of all conical limit points. Hence $\Lambda
(\Gamma)$ is the disjoint union 
$$\Sigma ({\H}^{n};\Z)\cup \{\text{\rm parabolic fixed points}\}$$   
\end{thm}

\begin{proof} By Definition $(GF2)$ of \cite{Bo93} we know that $\Lambda
(\Gamma)$ is the disjoint union of the conical limit points and the
bounded parabolic fixed points. Since $\Sigma ({\H}^{n};\Z)\subseteq \Lambda
(\Gamma )$, Lemma \ref{parabolic} implies what is claimed.  
\end{proof}

\begin{thm}\label{gf} $\Sigmacirc ({\H}^{n}; {\Z})$ is
non-empty if and only if $\Gamma $ is geometrically finite and has no
parabolic elements.
\end{thm}
\begin{proof}
Assume that $\Sigmacirc (\Gamma ;{\Z})$ is non-empty. By Theorem \ref{alternatives2} it must contain at least two points. Then the proof of Theorem \ref{alternatives2} implies that $C\cap {\H}^n$ is cocompact, where $C$ is the convex hull of $\Lambda (\Gamma)$. In \cite{Bo93} $\Gamma \backslash C$
is called the {\it convex core} of the orbifold $\Gamma \backslash {\H}^{n}$;
and we use \cite{Bo93} to deduce from the compactness of this convex
core that $\Gamma $ is geometrically finite (see Definition GF4 of
\cite{Bo93}.) By Theorem \ref{alternatives2} we also know that
$\Sigma ({\H}^{2};{\Z})=\Lambda (\Gamma)$, hence by Theorem \ref{conical1}
$\Gamma $ contains no parabolic element.

Conversely, if $\Gamma $ is geometrically finite and contains no
parabolic element then Theorem \ref{conical1} asserts that  $\Sigma
({\H}^{n};{\Z})=\Lambda (\Gamma)$. Since  $\Lambda (\Gamma)$ is
$G$-invariant and closed, we infer $\Sigmacirc ({\H}^{n};{\Z})=\Lambda
(\Gamma)\neq \emptyset$.  
\end{proof}	

Combining Corollary \ref{finite type1} and Theorem \ref{gf} we get:
\begin{cor}If $\Gamma $ is geometrically finite and has no parabolic elements
then $\Gamma $ has type $F_{\infty}$. 
\end{cor}

\subsection{Two examples of $\Sigma ({\H}^2;A)$ with $A$ a non-trivial $G$-module}\label{H2}

Let $p$ be a natural number $>1$, and let $G$ be the group of matrices of the form 
\[ \left[ \begin{array}{cc}
p^{k} & b \\
0 & p^{-k} \\
\end{array} \right]\] 
 where $k\in {\Z}$ and $b\in {\Z}[\frac{1}{p}]$.

Then $G$ is the semidirect product of the normal subgroup $N\cong
{\Z}[\frac{1}{p}]$ and the infinite cyclic group generated by

$t=\left[ \begin{array}{cc}
p & 0 \\
0 & p^{-1} \\
\end{array} \right]$  
(which acts on ${\Z}[\frac{1}{p}]$ by multiplication with $p^2$); i.e. $G\cong N\rtimes <t>$. This $G$ has the finite presentation
$$<a,t\mid tat^{-1}=a^{p^{2}}>$$

where $a$ represents the element $1\in {\Z}[\frac{1}{p}]$ or, equivalently, the matrix  
$\left[ \begin{array}{cc}
1 & 1 \\
0 & 1 \\
\end{array} \right]$.

We consider two $G$-modules $A$ and $B$, both having underlying abelian group ${\Z}[\frac{1}{p}]$ and trivial $N$ action; thus both are $<t>$-modules. The difference between them is that
\begin{itemize}
\item $t$ acts on $A$ by multiplication with $p^2$
\item $t$ acts on $B$ by multiplication with $p^{-2}$
\end{itemize}  
Both $A$ and $B$ are generated as $<t>$-modules by $a=1\in {\Z}[\frac{1}{p}]$, and as abelian groups by $\{\frac{1}{p^{2k}}\mid k\in {\N}\}$.

The group $G$ acts by M\"{o}bius transformations on the upper half plane model
of ${\H}^2$, and the point $\infty \in \partial {\H}^2$ is fixed under
this action. Moreover, $\infty $ is both a parabolic and a hyperbolic
fixed point; (each $n(b):=
\left[ \begin{array}{cc}
1 & b \\
0 & 1 \\
\end{array} \right]\in N$ 
is parabolic, while each $h\in G-N$ is hyperbolic.) The axis $L_0$ of the
generator $t$ connects $\infty $ with $0\in \partial {\H}^2$; the axis
of $n(b)tn(-b)$ is the line $n(b)L_{0}=:L_b$ which connects $\infty $
to $b\in {\Z}[\frac{1}{p}]$. The end points of these axes are conical
limit points, hence $\Sigma ({\H}^2;{\Z})$ contains the dense subset
${\Z}[\frac{1}{p}]\cup \infty$, and hence the argument in the proof of
Lemma \ref{parabolic}{\it (b)} shows that $\Sigmacirc ({\H}^{2};{\Z})=
\Sigma ({\H}^{2};{\Z})=\partial {\H}^{2}$.

To find $\Sigma ({\H}^{2};A)$ we consider the $G$-homomorphism
$\epsilon : {\Z}G\twoheadrightarrow A$ given by $\epsilon (1_{G})=1\in
A={\Z}[\frac{1}{p}]$. Thus $\epsilon (n)=1$ for $n\in N$, and $\epsilon
(t)=p^2$. We choose the control function $h:{\Z}G\to f{\H}^2$ which
maps every $g\in G$ to the singleton $\{g.i\}$. For a fixed $k\in {\Z}$
we have $a=p^{2k}\in A$, while the elements $c_{j}:=p^{2(k+j)}t^{-j}\in
{\Z}G$ lie in $\epsilon ^{-1}(a)$ for all $j\in {\N}$ such that $j\geq
k$. Their images under $h$ are $$h(c_{j})=h(t^{-j})=\{p^{-2j}.i\}\subseteq
{\H}^{2},$$ so that $0\in \partial {\H}^2$ is exhibited as an accumulation
point of $\epsilon ^{-1}(a)$ in the horospherical sense. As the control
images $h(c_{j})$ are singletons, the argument used in the proof of
Lemma \ref{parabolic}{\it (b)} applies in this situation and shows,
not only that the $G$-orbit of $0\in {\R}$ (i.e. set of $p$-adic
rationals ${\Z}[\frac{1}{p}]\subseteq {\R}$) lies in $\Sigma ({\H}^{2};A)$,
but also that each $r\in {\R}\subseteq \partial {\H}^{2}$ lies in $\Sigma
(G;A)$. The argument fails only for the boundary point  $\infty \in
\partial {\H}^2$ Indeed, we have: 
\begin{prop} 
\begin{enumerate}[(i)] 
\item $\Sigma ({\H}^{2};A)=\partial {\H}^2-\{\infty\}$
\item $\Sigmacirc ({\H}^{2};A)=\emptyset$ 
\item $\Lambda ({\H}^{2};A)=\partial {\H}^2$ 
\end{enumerate}
\end{prop}
\begin{proof}Assume we have found $c\in {\Z}G$ such that $\epsilon
(c)=1\in A$ and $h(c)\subseteq \{z\in {\C}\mid \text {Im}(z)\geq q\}$
for some large number $q$. Let $g=
\left[ \begin{array}{cc} p^{k} & b
\\ 0 & p^{-k} \\ \end{array} \right] \in G$ be a group element which
occurs in the support of $c$. Thus $h(g)=p^{2k}i+p^{k}b\in h(c)$. We
note that $t^{k}-g\in {\Z}G$ is in the kernel of $\epsilon $ and
that $h(t^{k}-g)=\{p^{2k}.i, p^{2k}.i+p^{k}b\}$. This shows that we
can replace $c$ by the element $c':=c+m(t^{k}-g)$ without losing the
two conditions $\epsilon (c')=1$ and $h(c')\subseteq \{z\in {\C}\mid
\text{Im}\geq q\}$. Repeating this procedure eventually replaces $c$
by a modified element $c''\in {\Z}G$ with $c''=\Sigma m_{k}t^{k}$, with
$\epsilon (c'')=1$ and $h(c'')=\{z\in {\C}\mid \text {Im}(z)\geq q\}$.
Now $h(c'')=\{p^{2k}.i\mid m_{k}\neq 0\}$, hence $p^{2k}\geq q$ for all
$k$ with $m_{k}\neq 0$. When $q>1$ this implies $k>0$; but $\epsilon
(c)=\Sigma _{k=1}^{l}m_{k}p^{2k}$ is divisible by $2p$, hence $\epsilon
(c)\neq 1$. This shows that $\infty$ is not a horospherical limit point
of $1\in A$.
\end{proof}
\begin{prop} 
$\Sigmacirc ({\H}^{2};B)=\Sigma ({\H}^{2};B)=\Lambda ({\H}^{2};B)=\{\infty \}$
\end{prop}
\begin{proof} Here the augmentation map $\epsilon :{\Z}G\twoheadrightarrow B$ is given by $\epsilon (n)=1\in B={\Z}[\frac{1}{p}]$ and $\epsilon (t)=p^{-2}$. We consider the $G$-map $\varphi :{\Z}G\to {\Z}G$ given by right multiplication with $p^{2}t$. Since $\epsilon (p^{2}t)=p^{2}p^{-2}=1$, $\varphi $ lifts the identity $\text{Id}_B$. Moreover we see that passing from $h(\lambda )$ to $h(\lambda p^{2}t)$ amounts to multiplying the imaginary part of each point $z\in h(\lambda )$ by $p^2$. Hence $\varphi $ pushes all control images towards $\infty $, i.e. $\infty \in \Sigmacirc ({\H}^{2};B)$. This shows that 
$\{\infty \}\subseteq \Sigmacirc ({\H}^{2};B)\subseteq \Sigma ({\H}^{2};B)=\Lambda ({\H}^{2};B)$.

For the opposite inclusion we note that the $G$-orbit $Ge$ is dense in $\partial {\H}^2$ for each $e\in \partial {\H}^{2}-\{\infty \}$. Hence if $\Lambda ({\H}^{2};B)$ contained a point $e\neq \infty$ it would contain all of $\partial {\H}^2$. An argument similar to the one showing that $\infty \notin \Sigma ({\H}^{2};A)$ shows that $0\notin \Sigma ({\H}^{2};B)$. The details are omitted.
\end{proof}


\newpage
\renewcommand{\appendixtocname}{Appendix}
\section*{Appendix: The connection with (tropical) algebraic geometry: \\$\Sigma (G\otimes {\R};A)$ when $G$ is abelian}\label{tropical}

$\;\;\;\;\;\;\;\;\;\;\;\;\;\;\;\;\;\;\;\;\;\;\;\;\;\;\;\;\;\;\;\;\;\;\;\;\;\;\;\;\;$ by Robert Bieri

\renewcommand{\thesubsection}{\Alph{subsection}}

\setcounter{subsection}{0}
\numberwithin{thm}{subsection}
\numberwithin{equation}{subsection}

In Section \ref{translation} we have seen that when
$M=G_{\text{ab}}\otimes {\R}$ is the canonical Euclidean $G$-space
then the horospherical limit set $\Sigma(M;A)$ coincides with $\Sigma
^{0}(G;A)$, the $BNSR$-geometric invariant of the pair $(G,A)$ in
dimension zero. To compute $\Sigma (G_{\text{ab}}\otimes {\R};A)=\Sigma
^{0}(G;A)$  is still not an easy matter. The direct approach, based
solely on the definition, is only successful in specific examples; and
no general alternative method is known --- except in the case when the
group $G$ is abelian. In that case the group ring ${\Z}G$ is commutative
and Noetherian so that powerful methods from commutative algebra are
available. That is the background of this Appendix.

When $G$ is an abelian group $\Sigma (G\otimes {\R};A)=\Sigma ^{0}(G;A)$  is
really the older geometric invariant which was introduced in \cite{BS80},
while the powerful method to compute it together with its geometric
consequences was established only in \cite{BGr84} --- see Part C of Subsection \ref{subsection}.

Part \ref{global} of this Appendix is a review of the main results of that
paper: It establishes that the complement of $\Sigma ^{0}(G;A)$ in $\del
{\R}^n$ is the image of a polyhedral subset $\Delta \subseteq {\R}^n$
under the radial projection\footnote{Throughout, whenever we apply $\rho $ to a set it is understood that the point $0$ is to be removed from that set first.} $\rho :{\R}^{n}-\{0\}\to \del {\R}^n$. In the
late nineties $\Delta$ was recognized (e.g. \cite{St02}, \cite{EKL06} and \cite{MS15}) as the integral
version of the tropical variety associated to the annihilator ideal $I$ of
$A$. See Corollary \ref{radial} below.

Part B is concerned with the result of Einsiedler, Kapranov, and Lind
\cite{EKL06} establishing a direct connection between $\Delta $ and the
affine algebraic variety $V$ of $I$. In their language $\Delta $ is the
non-Archimedean part of the adelic amoeba of $V$. 

\subsection {$\Sigma (G;A)$ in terms of the global tropical variety of the annihilator ideal of $A$}\label{global}

\subsubsection{\bf The set-up}

In this Appendix $G$ will always denote a finitely generated abelian group, and we will write $\Sigma (G;A)$ rather than $\Sigma ^{0}(G;A)$.

By a {\it character} on $G$ we mean a homomorphism $\chi :G\to {\R}$ into the additive group of real numbers, and we write $G_{\chi }\subseteq G$ for the submonoid $\{g\in G\mid \chi (g)\geq 0\}$. Two characters $\chi $ and $\chi '$ are {\it equivalent}
if $\chi =r\chi '$ for some positive constant $r\in {\R}$. We
write $[\chi]={\R}_{> 0}\chi $ for the equivalence class of $\chi$,
$\del \text{Hom}(G,{\R})$ for $\{[\chi ] \mid 0\neq \chi \in \text{Hom}(G,{\R})\}$,
and $\rho :\text{Hom} (G,{\R})-\{0\}\to \del \text{Hom} (G,{\R})$  for the {\it radial projection} $\rho (\chi )=[\chi ]$.

The affine $G$-space $\text{Hom} (G,{\R})$ is isomorphic to $M=G\otimes
{\R}\cong {\R}^n$, and the half-spaces and corresponding filtrations of
$M$ can be described by characters $\chi $ on $G$ and their rays $[\chi]$
without reference to an inner product on $M$. Thus we can interpret the
horospherical limit set $\Sigma (G;A)$ as a subset of the sphere $\del \text{Hom}(G,{\R})$,
with the advantage that we can ignore the metric on $M$ when we study
the functorial properties of $\Sigma (M;A)$ with respect to the group argument.

Some of the computational difficulties with $\Sigma (G;A)$
disappear when the base ring $\Z$ is replaced by a field. In order to
cover both cases throughout, we let $D$ denote a Noetherian domain and we assume that $A$ is a finitely generated $DG$-module. Then we consider the (open) subset of the sphere $\del \text{Hom}(G,{\R})$

\begin{equation}\label{A.0}
\Sigma _{D}(G;A)=\{[\chi ]\mid A \text{ is finitely generated as a } 
DG_{\chi }\text{-module}\}.
\end{equation}

Thus $\Sigma _{{\Z}}(G;A)=\Sigma (G;A)$ --- see Part A of Subsection \ref{subsection}.

It is often more convenient to work with the complement of $\Sigma
_{D}(G;A)$ in $\del \text{Hom}(G,{\R})$; we denote it by $\Sigma
_{D}(G;A)^c$.

\subsubsection{\bf $\Sigma _{D}(G;A)$ in terms of valuations}\label{computation}

The first step in computing $\Sigma _{D}(G;A)$ was available
at the time of \cite{BS80} and reappears as the special case of our
Corollary \ref{initial} where $M$ is the canonical Euclidean $G$-space
$G\otimes {\R}$. As the translation action of $G$ on $M$ induces the
trivial action on $\del M$, that corollary recovers Formula (2.3) of \cite{BS80}
which asserts

\begin{equation}\label{A.1}
\Sigma _{D}(G;A)=\{[\chi ]\mid \text{ there exists }\lambda \in DG \text{ with }\lambda A=0 \text{ and }\lambda _{\chi }=1\}
\end{equation}

\noindent where $\lambda _{\chi }$ stands for the {\it initial part of} $\lambda $,
collecting the monomials $mg$ of $\lambda $ with minimal $\chi (g)$. This
shows, in particular, that $\Sigma _{D}(G;A)=\Sigma _{D}(G;DG\slash
I)$ where $I = \text{Ann}_{DG}(A)$ denotes the annihilator ideal of $A$ in the group algebra $DG$.

By (\ref{A.1}) we can (and did) compute $\Sigma _{D}(G;A)$ when the
annihilator ideal of $A$ is principal. But to compute it in more general
situations requires a new ingredient: valuation theory\footnote{Theorem
1.2 of \cite{BS81} provided a first description of the complement 
$\Sigma _{D}(G;DG\slash I)^{c}$ in terms of valuations on $R$; and Theorem 8.1 of
\cite{BGr84} improved this result by showing, (\ref{complement}) below,
that only real valuations on $R$ are needed.} on the commutative ring
$R=DG\slash I$.

Here are the basic facts on valuations. We write ${\R}_{\infty}$ for ${\R}\cup \{\infty\}$. By a (non-Archimedean real) {\it valuation} on a commutative ring with unity $R$ we mean a map $w:R\to {\R}_{\infty}$ with the properties
$$w(0)=\infty, w(1)=0, w(ab)=w(a)+w(b), w(a+b)\geq \text {min}\{w(a), w(b)\}$$
for all $a,b \in R$. 

We write $\text{val}(R)$ for
the set of all valuations on $R$. Two valuations on $R$ are {\it equivalent}
if they coincide up to multiplication by a positive real constant. If
$J$ is an ideal in $R$ then composition with the canonical projection
identifies $\text{val}(R\slash J)$ with $\{w\in \text{val}(R)\mid
w(J)=\infty\}$. The {\it center}, $w^{-1}(\infty)$, of the valuation $w$
is always a prime ideal in $R$; in particular, it cannot contain a unit of
$R$. The valuation $w$ is {\it centerless} if $w^{-1}(\infty)=\{0\}$. Each valuation $w$ on $R$ factorizes canonically via a unique centerless valuation $w'$ on $R\slash w^{-1}(\infty )$. The centerless valuation which only takes on the values $0$ and $\infty $ will be denoted by $0$.

As above, we take $R =DG\slash I$ where $I$ is the annihilator ideal of the $DG$-module $A$. The image of
$G$ under the canonical projection $\kappa :DG\twoheadrightarrow R$ is a
group of units of $R$; hence $w(\kappa (G))\subseteq {\R}$ for every $w\in
\text{val}(R)$.  Each valuation $w$ on $R$ induces a character on $G$,
$\chi _{w}:=w\kappa |_{G}:G\to {\R}$ and a valuation on $D$ $w\kappa|_{D}:D\to
{\R}_\infty$. By Theorem 8.1 of \cite{BGr84} every element of $\Sigma
_{D}(G;R)^{c}$ is represented by a character $\chi _{w}$ on $G$ induced by
valuations on $R$ with $w(D)\geq 0$\footnote{This should be $w\kappa$, but from now on we drop the $\kappa $ for ease of notation.}, while the converse was already observed in [BS81],
i.e. we have \footnote{George Bergman
\cite{Be71} proved that if $D=k$ is a field then the right hand sides of
(\ref{A.1}) and (\ref{complement}) coincide, and he conjectured that
this is a polyhedral set. Therefore $\Sigma _{k}(G;kG\slash I)^{c}$ is called the {\it Bergman fan} of $I$.}

\begin{equation}\label{complement}
\Sigma _{D}(G;R)^{c}=\rho (\{w\kappa |_{G}\mid w\in \text{val}(R), w\kappa (D)\geq 0\}). 
\end{equation}

\noindent In simple cases when all valuations on $R$ are known --- e.g., when $R$
is a ring of rational numbers --- (\ref{complement}) computes $\Sigma _{D}(G;R)^{c}$ immediately. In that sense (\ref{complement}) is a more powerful tool than (\ref{A.1}).

\subsubsection{\bf The ``local" tropical variety $\Delta ^{v}(G;K)$ of the ideal $I\leq DG$}\label{secA2}

Let $\kappa :DG\thra R=DG\slash I$ be as above. Given a valuation $v:D\to
{\R}_{\infty}$ with $\text{ker}(\kappa |D)\subseteq v^{-1}(\infty )$ and a
commutative ring $K\supseteq R$ containing $R$ as a subring, we consider
the subset of $\text{Hom}(G,{\R})$,

\begin{equation*}
\Delta ^{v}(G;K):=\{w\kappa |_{G}\mid w\in \text{val}(K) \text{ with }w\kappa|_{D}=v\}
\end{equation*}

\noindent (Usually $K$ will be $R$ itself.) Then  (\ref{complement}) breaks up as 

\begin{equation}\label{union}
\Sigma _{D}(G;R)^c=\bigcup _{v}\rho (\Delta ^{v}(G;R))
\end{equation}
\noindent where $v$ runs through all valuations on $D$ with $v(D)\geq 0$.

It turns out that the sets $\Delta ^{v}(G;R)$ are the key to the
geometry of $\Sigma _{D}(G;R)^c$; they are much better behaved than
their images $\rho (\Delta ^{v}(G;R))$ under the radial projection. In
fact they have some excellent geometric properties which, although
distorted under radial projection, still impose restrictions on
$\rho (\Delta ^{v}(G;R))$ and hence on $\Sigma _{D}(G;R)^c$. Later
developments confirmed the importance of the $\Delta ^{v}(G;R)$: they reappeared
--- independent of \cite{BGr84} but enhanced with the computational power
of Gr\"{o}bner basis techniques --- as the central objects of Tropical
Geometry\footnote{Tropical Geometry emerged in the early nineties as a
systematic attempt to investigate the analogue of algebraic geometry over
the {\it tropical semi-ring} $({\R}, \oplus, \odot)$, the real numbers
with addition $a\oplus b=\text { min}\{a,b\}$ and $a\odot b=a+b$. See
for example \cite{St96} or \cite{RST05}.  By the turn of the
millenium it had mellowed from a $PL$ analogue of algebraic geometry
into a powerful tool for applications in classical algebraic geometry; see
\cite{M05}, \cite{MS15}.}. We call $\Delta ^{v}(G;R)$ {\it the tropical variety}\footnote{In the
literature the tropical variety of a polynomial ideal $I$ is usually
defined as the intersection of the singularity locus of the tropical
polynomials associated to the elements of $I$. The equivalence of the
two definitions drops out as a byproduct of the proof of the main results in 
\cite{BGr84}; see Section \ref{proofs} } {\it of the ideal $I$ with respect to the valuation $v$.}

A closed subset $\Delta \subseteq {\R}^n$ is {\it polyhedral}
if it is a finite union of finite intersections of closed half spaces;
it is {\it rationally defined} over the additive subgroup $S\leq {\R}$
if the half-spaces are defined by linear equations of the form $f(x)=s$ where $f(x)$ has rational coefficients and $s\in S$. Here is the first main result on $\Delta ^{v}(G;R)$:

\begin{thm}\cite{BGr84}\label{polyhedral1}The subset $\Delta ^{v}(G;R)\subseteq \text{\rm Hom}(G,{\R})$ is polyhedral and rationally defined over the additive group generated by ${\Z}$ and the finite values in $v(D)$.
\end{thm}

{\bf $\Delta ^{v}(G;R)$ in terms of valuations on fields}

Each valuation $w$ on $R$ with $w\kappa \mid _{D}=v$ factorizes via
$R\thra R\slash Rv^{-1}(\infty )$ and hence via some $R\slash RP_i$,
where $\{P_{1},\dots P_{m}\}$ is the set of prime ideals in $R$ which
are minimal over $Rv^{-1}(\infty )$. This shows that

\begin{equation*}
\Delta ^{v}(G;R)=\bigcup _{i=1}^{m}\Delta ^{v'}(G;R\slash P_{i})
\end{equation*}

\noindent and thus reduces computation to the case where $R$ is a domain. Thus computation is reduced to the case where $R$ is a domain,
$D\subseteq R$ and $v:D\to {\R}_{\infty}$ is centerless.

In this situation we can refer to Theorem 6.1 of \cite{BGr84} which asserts that if $w:R\to {\R}_{\infty}$ is a valuation on $R$ which extends the centerless valuation $v$ then there is a centerless valuation $w':R\to {\R}_{\infty}$, extending $v$, with 
$w\kappa \mid _{G}= w'\kappa \mid _{G}$. Hence
\begin{equation*}
\Delta ^{v}(G;R)=\Delta ^{v}(G;K)
\end{equation*}
\noindent when $v$ is centerless and $K$ is any field that contains the domain $R$.

\subsubsection{\bf The case when $D$ is a Dedekind domain}

When $D$ is a Dedekind domain every non-trivial prime ideal of $D$
is maximal, the non-centerless valuations $v$ on $D$ with $v(D)\geq 0$
have their values in $\{0,{\infty}\}$ and hence have singleton equivalence classes;
the same holds for the trivial valuation on $D$. Any other non-negative
valuation $v$ on $D$ is equivalent to one of the familiar ${\mathfrak
p}$-adic valuations $v_{\mathfrak p}$ associated to the non-trivial prime
ideals ${\mathfrak p}$, and given by the exponent of $\mathfrak p$ in the
unique prime decomposition of the ideals of $D$. Thus, associated to every
non-trivial centerless valuation $v$ on $D$ is a ${\mathfrak p}$-adic
valuation $v_{\mathfrak p}$ equivalent\footnote{Two valuations are {\it equivalent} if one is a positive real multiple of the other.} to $v$, and a non-centerless
valuation $\widehat v$ with ${\mathfrak p}={\widehat v}^{-1}(\infty)$.

It turns out that when $R=DG\slash I$ is a domain containing $D$ then the tropical varieties 
$\Delta ^{{\widehat v}}(G;R)$ and $\Delta ^{0}(G;R)$ are determined by, and determine, the local behavior of $\Delta ^{v}(G;R)$ at $0$ and $\infty $. To explain this we introduce some terminology on a polyhedral subset $\Delta $ of the affine space
${\R}^n$. To describe the behavior of $\Delta $ in the neighborhood of a
point $x\in \Delta $ we consider the union of all rays $[x,e) \subseteq
{\R}^n$ which emanate from $x$ by starting with an initial segment of
positive length in $\Delta $. We call this the {\it local cone} at $x$
and we denote it by $\text{LC}_{x}(\Delta)$; its boundary at infinity
$\del \text{LC}_{x}(\Delta)$, is also called the {\it link of} $x$ {\it in} $\Delta $, 
denoted $\text{\rm lk}_{\Delta}(x)$. To
describe $\Delta $ in the neighborhood of infinity, we consider the
union of all rays $[0,e)$ emanating from the base point $0\in {\R}^n$
with the property that $e\in \del {\R}^n$ is represented by a ray in $\Delta $. We
call this the {\it local cone} of $\Delta $ {\it at infinity} and we
denote it by $\text{LC}_{\infty}(\Delta)$. Its boundary at infinity,
$\del (\text{LC}_{\infty}(\Delta))$, coincides with $\del \Delta$, the
set of all equivalence classes of rays in $\Delta $. 

\begin{thm}\label{local}(\cite{BGr84}) Let $D$ be a Dedekind domain endowed with a non-zero centerless valuation $v$ and embedded in the domain $R=DG\slash I$. Then we have 

$\Delta ^{\widehat v}(G;R)={\rm LC}_{0}(\Delta ^{v}(G;R))$
 
and

$\Delta ^{0}(G;R)=\text{\rm LC}_{\infty}(\Delta ^{v}(G;R))$.
\end{thm}

Theorem \ref{local} shows, in particular, that the sets $\rho (\Delta
^{\widehat v}(G;R))$ are redundant in the union on the right hand side of
(\ref{union}), and since the radial projection $\rho (\Delta ^{v}(G;R))$
depends only on the equivalence class of $v$, we find that (\ref{union})
reduces to

\begin{equation}\label{specunion}
\Sigma _{D}(G;R)^{c}=\rho (\bigcup _{{\mathfrak p}\in \text{Spec}(D)}\Delta ^{v_{\mathfrak p}}(G;R))
\end{equation}

The next theorem shows that almost all of the sets $\Delta ^{v_{\mathfrak p}}(G;R)$ are redundant:

\begin{thm}(\cite {BGr84})\label{finiteset} With $D$ as in Theorem \ref{local}, there is a finite set $S\subseteq \text{\rm Spec}(D)$ with the property that 
$$\Delta ^{v_{\mathfrak p}}(G;R)=\Delta ^{0}(G;R)$$
\noindent for all prime ideals ${\mathfrak p}\notin S$.
\end{thm}

As a consequence we find that $$\Delta ^{D}(G;R):= \bigcup _{{\mathfrak
p}\in \text{\rm Spec}(R)}\Delta ^{v_{\mathfrak p}}(G;R)$$ \noindent is
in fact a finite union. Hence $\Delta ^{D}(G;R)$ is polyhedral and is
a tropical variety. We call it {\it the global tropical variety of the
ideal} $I=\text{Ann}(A)$. We have now reached the goal of expressing
our invariant in terms of the radial projection of a tropical variety:

\begin{cor}\label{radial} If $D$ is a Dedekind domain, $A$ a finitely generated $DG$-module, and $R=DG\slash \text{\rm Ann}_{DG}(A)$, then 
\begin{equation}
\Sigma _{D}(G;A)^{c}=\rho (\Delta ^{D}(G;R))
\end{equation}
\noindent and is a rational polyhedral subset of the sphere $\del \text{Hom}(G,{\R})$.
\end{cor}

\subsubsection{\bf Pure dimension and balanceability of $\Delta ^{v}$}

Here $D$ is an arbitrary domain endowed with a centerless valuation
$v:D\to {\R}_{\infty}$ and embedded in the domain $R=DG\slash I$. If
$k=\text{Frac}(D)$ denotes the field of fractions of $D$ we can extend
$v$ to $k\to {\R}_{\infty}$ and replace $R$ by $k\otimes _{D}R$. Thus
we may assume without loss of generality that $k\subseteq R$. We write
$\text{tr.deg}_{k}R$ for the transcendence degree of $R$ over $k$ (the
maximum number of algebraically independent elements of $R$ over $k$).

Balanceability is a local property that a polyhedral subset $\Delta \subseteq {\R}^n$ may or may not possess at a point $x\in \Delta $. We say that $\Delta $ is {\it balanceable} at $x\in \Delta $ if the convex hull of $\text{LC}_{x}(\Delta)$ is an affine subspace of ${\R}^n$; in other words, when the convex hull of the link $\text{\rm lk}_{\Delta }(x)$ is a subsphere of $\del {\R}^n$.

\begin{thm}\cite{BGr84}\label{dimension} For every valuation $v:k\to {\R}_{\infty}$ and every domain $R$ the polyhedral set $\Delta ^{v}(G;R)$ has the following properties:
\begin{enumerate}[\rm (i)]
\item $\Delta ^{v}(G;R)$ is of pure dimension equal to $\text{\rm tr.deg}_{k}R$;
\item $\Delta ^{v}(G;R)$ is balanceable at every point $x\in \Delta ^{v}(G;R)$.
\end{enumerate}
\end{thm}

As  ``pure dimension", ``balanced", and the ``local cone" constructions
carry over to finite unions, the assertions of Theorem \ref{dimension}
carry over, {\it mutatis mutandis}, to $\Delta ^{D}(G;R)$ when $D$
is a Dedekind domain. Moreover, if $D=k$ is a field then $\Delta
^{D}(G;R)=\Delta ^{0}(G;R)$ is a cone based at $0$, and in that situation
those properties are also preserved under radial projection. Hence we have

\begin{cor} At each of its points $x$, $\Sigma _{k}(G;R)^c$ is balanced \footnote{For a polyhedral subset $L$ of a sphere, the link $\text{lk}_{L}(x)$ of $x$ in $L$ and the property ``balanced at $x$" are to be interpreted in the tangent space 
at $x$.} and of pure dimension equal to $\text{\rm tr.deg}_{k}R-1$.
\end{cor} 

It is obvious that the conical set $\Delta ^{0}(G;R)^{c}$ cannot be balanceable if it is contained in an open half space and is non-zero. Hence we find that when $R$ is a domain and $\Sigma _{k}(G;R)^{c}$ lies in an open hemisphere then  $\Sigma _{k}(G;R)^{c}$ is empty and $R$ is algebraic over $k$.

When $D$ is not a field, $\Sigma _{k}(G;R)^c$ is, in general, neither
of ``pure dimension" nor ``balanced" at each of its points. But the
failure of each of these properties is well understood (see Sections
8.6 and 8.7 of \cite{BGr84}), so similar applications are available.

\subsubsection{\bf Remark on the proofs}\label{proofs}

After the reduction to the case
where $R=DG\slash I$ is a domain containing the field of fractions
$k=\text{Frac}(D)$, the proofs of Theorems \ref{polyhedral1} through \ref{finiteset} given in
\cite{BGr84} can be outlined in a nutshell: Given  a valuation $v$
on $k$, and an irreducible polynomial $f(X)\in k[X]$, the set of
all slopes of the Newton polygon of $f(X)$ coincides with the set
of all values that an extension of $v$ can achieve on the roots of
$f(X)$. This old and well-understood fact makes it easy to work out
(and prove polyhedrality of) $\Delta^{v}(G;R)$ when $I$ is a principal
ideal. It works equally well in the somewhat more general case when $G$
is generated by a set $Y\cup \{\alpha \}$ where $Y$ is a transcendence
basis of $K=\text{Frac}(R)$ over $k$, and $\alpha $ is algebraic over
$k(Y)$.  
 
The naive idea of iterating the Newton polygon
argument and working out all possible values on the generators of the
group $G$ does yield a polyhedral set containing $\Delta ^{v}(G;R)$.
But in order to compute $\Delta ^{v}(G;R)$ itself one would have to
pin down which of the possible values on the generators can be achieved
simultaneously; i.e., one would have to impose the algebraic dependence
of the generators. The method of \cite{BGr84} to get around such difficult
questions, probably the main contribution of that paper, can be described
as ``Tomography by Generic Projections". The idea is to show that $G$
contains a ``dense" set of subgroups $H=\text{gp}(Y,\alpha )$, generated
by a transcendence basis $Y$ of $K$ together with one additional element
$\alpha $ which is algebraic over $k(Y)$. By the Newton polygon argument,
above, one can compute the sets $\Delta ^{v}(H;R)$ and observe that
they are polyhedral. Then it remains to observe that the restriction
map $\text{res}_{G,H}:\text{Hom}(G,{\R})\to \text{Hom}(H,{\R})$ yields
a surjection $\Delta ^{v}(G;R)\twoheadrightarrow \Delta ^{v}(H;R)$,
and that one can always find a finite number of such subgroups $H$ to
exhibit $\Delta ^{v}(G;R)$ as a finite intersection of polyhedral sets
$\text{res}_{G,H}^{-1}(\Delta ^{v}(H;R))$.

The tropical geometers' proof that their varieties are polyhedral is based on Gr\"{o}bner bases. That proof is independent of, and shorter than, the original one and readily provides algorithms for explicit computations. On the other hand, the Gr\"{o}bner basis algorithms are less compatible with the geometry of $\Delta ^{v}(G;R)$ than the original technique based on exhibiting $\Delta ^{v} (G;R)$ in terms of computable images under restriction maps. In fact, \cite{BGr84} establishes and applies the tomography by generic projections as a general method to reduce questions on tropical varieties to the principal ideal case.

\subsection{Connection with algebraic geometry}

\subsubsection {\bf From the tropical to the algebraic variety}

It is extremely rewarding to interpret the polyhedral sets $\Delta
^{v}(G;R)$ in the light of algebraic geometry. Here $k$ is a field endowed with a valuation $v:k\to {\R}_{\infty}$, 
and we consider the special case when the
group $G$ is free abelian of rank $n$ with a specified basis ${\mathbf
X}=\{X_{1},\dots,X_{n}\}$, so that the group algebra $kG$ is the
Laurent polynomial ring $k[{\mathbf X}]$. The field $k$ is embedded in its
algebraic closure $\bar k$ and we assume that $\bar k$ is endowed with a
non-trivial valuation ${\bar v}:{\bar k}\to {\R}_{\infty}$ extending $v$. In this situation
we consider, along with an ideal $I\leq kG$, its toric variety $V=V(I)\subseteq 
({\bar k}-{0})^n$.

For simplicity we assume throughout that $V$ is irreducible, i.e.
$I$ is a prime ideal. As above, $R=kG\slash I$ and $K=\text{Frac}(R)$. By Theorem \ref{dimension} we already know that $\Delta ^{v}(G;R)=\Delta ^{v}(G;K)$ (see Section \ref{secA2}) is of pure dimension equal to the dimension of $V$; but now we can mate the variety $V$ with the polyhedral
set $\Delta ^{v}(G;K)$ more intimately by means of the natural maps

\begin{equation}\label{A.7}
\tau _{\alpha} :V\to \text{Hom}(G,{\R})
\end{equation} 

\noindent which are defined for each additive character $\alpha :{\bar k}-\{0\}\to {\R}$ as follows: 
Given $z\in V$, $\tau_{\alpha}(z):G\to {\R}$ is the composite map
$$
\xymatrix{
&G\ar[r]^{\kappa |_{G}} &R\ar[r]^{\text{ev}_{z}} &{{\bar k}-\{0\}}\ar[r]^{\alpha } &{\R}\\
}
$$
\noindent where $\kappa :kG\thra R$ is the canonical projection, and
$\text{ev}_{z}:R\to {\bar k}$ is evaluation of the rational functions on
$V$ at $z\in V$. We note that when $\alpha = {\bar v}$ is a valuation
extending $v$ then ${\bar v}\circ\text{ev}_{z}:R\to {\R}_{\infty}$
is a valuation on $R$, and hence $\tau _{\bar v}(V)\subseteq \Delta
^{v}(G;K).$ The important result here is due to Einsiedler,
Kapranov and Lind \cite{EKL06} and asserts

\begin{thm}(\cite{EKL06})\label{equal} The image of $\tau _{\bar v}$ is dense in $\Delta ^{v}(G;K)$.
\end{thm}

In particular $\text{\rm cl}(\tau _{\bar v}(V))$ depends only on $v$ and not on the particular extension ${\bar v}$ to 
${\bar k}$. Note that composing $\tau _{\alpha }$ with the evaluation-at-${\mathbf X}$ isomorphism
$\text{Hom}(G,{\R})\cong {\R}^n$ recovers the componentwise evaluation $\alpha :V\to {\R}^n$, and identifies 
$\tau _{\alpha }(V)$ with $\alpha (V)\subseteq {\R}^n$. Einsiedler, Kapranov and Lind define the {\it tropicalization}
of $V$ to be ${\mathcal T}_{\bar v}(V):=\text{\rm cl}{\bar v}(V)\subseteq {\R}^n$, and they use  the identification
$\Delta ^{v}(G;K)={\mathcal T}_{\bar v}(V)$, referring to Theorems \ref{polyhedral1} and \ref{dimension} for its properties. By Theorem \ref{local}, $\text{\rm
LC}_{\infty}(\Delta ^{v}(G;K))=\Delta ^{0}(G;K)$ for every valuation $v$;
and by definition of the local cone, $\rho (LC_{\infty}(\Delta))=\del
\Delta $. Hence, by (\ref{union}) with $D=k$, we find

$$\Sigma _{k}(G;R)^{c}=\rho (\Delta ^{0}(G;R))=\rho (\text{\rm
LC}_{\infty}(\Delta ^{v}(G;R))) =\del \Delta ^{v}(G;R).$$ 
\noindent
In other words 
\begin{cor}\label{independent} $\del {\mathcal T}_{\bar
v}(V)$ is independent of $v$ and coincides with $\Sigma _{k}(G;K)^{c}$.
\end{cor}

\subsubsection{\bf Amoebas}

Each of our valuations $v:k\to {\R}_{\infty}$ on the field $k$ can be transformed into a non-Archimedean norm on $k$ by putting 
$$|a|_{v}:=e^{-v(a)}, \;\;\;\;\;\;\; a\in k$$
which satisfies the strong (i.e. ultrametric) triangle inequality 

$$|a+b|_{v}\leq \text{max}(|a|_{v}, |b|_{v}) \;\;\;\;\;\;\;\;a,b\in k$$

Now $k$ and its algebraic closure $\bar k$ might carry other ---
Archimedean --- norms and it is often important to consider {\it all}
norms simultaneously.  Each norm $|.|:{\bar k}\to {\R}_{\geq 0}$ defines
an additive character $\text{\rm ln}|.|:{\bar k}-\{0\}\to {\R}$. Hence,
following \cite{EKL06}, we consider the set ${\mathcal A}(V):=-\text{\rm
cl}\tau _{\text{\rm ln}|.|}(V) \subseteq \text{Hom}(G,{\R})\cong {\R}^n$
for every norm $|.|$ on ${\bar k}$, and call this the {\it amoeba} of
$V$ with respect to $|.|$. If $|.|=|.|_v$ is given by a non-Archimedean
valuation $v$ we have ${\mathcal A}(V)=-\Delta ^{v}(G,K)$ by Theorem
\ref{equal}, and, by Theorem \ref{polyhedral1}, ${\mathcal A}(V)$
is polyhedral. As polyhedrality relies on the ultrametric triangle
inequality, the same behavior cannot be expected in the Archimedean
case. Indeed, when $V\subseteq {\C}^2$ is a generic complex algebraic
curve then ${\mathcal A} (V)\subseteq {\R}^2$ is a $2$-dimensional set
with differentiable frontier and finite limit set, the shape of which
explains the picturesque name that, along with the foundations of amoeba
theory, was introduced in \cite{GKZ94}.

In \cite{Be71} George Bergman used $\text{ln}|V|:=\{(\text{ln}|z_{1}|,\dots ,
\text{ln}|z_{n}|)\mid (z_{1},\dots , z_{n})\in V\}$ to define his {\it
logarithmic limit set} $V_{\infty}$, which can be interpreted as the
set of all cone-topology limit points\footnote{In keeping with the
picturesque language, one could say that Bergman's logarithmic limit
set of $V$ is the set of all directions of the tentacles of the amoeba
${\mathcal A}(V)$.} $V_{\infty}:=\Lambda (\text{ln}|V|)$.  As mentioned
in a footnote related to (\ref{complement}), Bergman conjectured that
$\Sigma _{k}(G;kG\slash I)^{c}$ is polyhedral, and he showed that if this
conjecture holds true it would imply $V_{\infty}=\Sigma _{k}(G;kG\slash
I)^{c}$. As Theorem \ref{polyhedral1} established Bergman's conjecture
we thus have the following supplement to Corollary \ref{independent}:

\begin{cor} $\Sigma _{k}(G;kG\slash I)^{c}$  coincides not only with the boundaries $\del
T_{\bar v}(V)$ of the tropicalizations of $V$ with respect to all valuations
$\bar v$ but also with the logarithmic limit set $V_{\infty}$.
\end{cor}

{\bf Remark:} Up to equivalence, the field $\Q$ of rational numbers admits, along
with all $p$-adic norms, only the usual absolute value. Hence, if $I$ is an ideal in the integral group ring ${\Z}G$, the adelic amoeba of $R={\Z}G\slash I$ considered in \cite{EKL06} is just the union of
the negative global tropical variety $-\Delta ^{\Z}(G;R)$
and the complex amoeba $\text{ln}|V(I)|$. By referring to \cite{ELMW01},
Einsiedler, Kapranov and Lind show that the adelic amoeba contains key
information on the dynamics of the action of $G$ on the Pontryagin dual
of a finitely generated $G$-module $A$ with $I=\text{Ann}_{G}(A)$:
the so-called ``non-expansive set of
this dynamical system'' is the radial projection of the adelic amoeba of
$R$. Hence, by (\ref{specunion}) the expansive set is the union of the negative
invariant $-\Sigma (G;A)^c$ and the complement $\rho (\text{ln}|V|)$
of the radial projection of the complex amoeba of the variety $V(I)$.


\def\cprime{$'$}
\providecommand{\bysame}{\leavevmode\hbox to3em{\hrulefill}\thinspace}
\providecommand{\MR}{\relax\ifhmode\unskip\space\fi MR }
\providecommand{\MRhref}[2]{%
  \href{http://www.ams.org/mathscinet-getitem?mr=#1}{#2}
}
\providecommand{\href}[2]{#2}

\end{document}